\documentclass{amsart}
\usepackage{graphicx}
\usepackage{amssymb}
\usepackage{amsmath}
\usepackage{latexsym}
\newtheorem{theorem}{Theorem}
\newtheorem{lemma}{Lemma}[section]
\newtheorem{proposition}{Proposition}[section]

\newtheorem{definition}{Definition}
\newcommand{\rn}{\mathbb{R}^n}
\newcommand{\nn}{\mathbb{N}}

\newcommand{\phiki}{\Phi_{k,i}}

\newcommand{\beq}{\begin{eqnarray*}}
\newcommand{\eeq}{\end{eqnarray*}}
\newcommand{\beqn}{\begin{eqnarray}}
\newcommand{\eeqn}{\end{eqnarray}}
\newcommand{\bequa}{\begin{equation}}
\newcommand{\eequa}{\end{equation}}

\title[Positivity and almost positivity of biharmonic Green's  functions]{Positivity  and almost positivity
of biharmonic Green's  functions 
under Dirichlet boundary conditions
}

\author{Hans-Christoph Grunau }
\address{Fakult\"at f\"ur Mathematik, Otto-von-Guericke-Universit\"at, Postfach 4120, 39016 Magdeburg, Germany}
\email{hans-christoph.grunau@ovgu.de}
\author{Fr\'ed\'eric Robert}
\address{Universit\'e de Nice-Sophia Antipolis, Laboratoire J.-A.Dieudonn\'e,
Parc Valrose, 06108 Nice Cedex 2, France}
\email{frobert@math.unice.fr}

\begin{document}

\maketitle

\begin{center}
{\em Dedicated to Prof. Wolf von Wahl on the occasion of his 65th birthday}
\end{center}

\begin{abstract}
In general, for higher order elliptic equations and boundary value
problems like the biharmonic equation and the linear clamped plate
boundary value problem  neither a maximum principle
nor a comparison principle or -- equivalently -- a positivity
preserving property is available. The problem is rather involved since the
clamped boundary conditions prevent the boundary value problem
{from} being reasonably written as a system of second order boundary value
problems.

It is shown that, on the other hand, for bounded smooth domains
$\Omega \subset\mathbb{R}^n$, the negative part of the corresponding
Green's function is ``small'' when compared with its singular positive
part, provided $n\ge 3$.

Moreover, the biharmonic Green's  function in balls $B\subset\mathbb{R}^n$
under Dirichlet (i.e. clamped) boundary conditions is known explicitly and is positive.
It has been known for some time that positivity is preserved
 under small regular perturbations
of the domain, if $n=2$. In the present paper, such a stability result is proved for $n\ge 3$.
\end{abstract}

\section{Introduction}

Although
 simple examples show that strong maximum principles as satisfied
e.g. by harmonic functions cannot hold true for solutions of
higher order elliptic equations, it is reasonable to ask whether
higher order \emph{boundary value problems} may possibly enjoy a
\emph{positivity preserving property}. To be specific, we consider
the clamped plate boundary value problem:
\begin{equation}\label{cpe}
\left\{
\begin{array}{c}
\Delta ^2u=f
\mbox{ in }\Omega ,\medskip\  \\
u_{\vert _{\partial \Omega }}=\frac
\partial {\partial \nu}u_{\vert _{\partial \Omega }}=0.
\end{array}
\right.
\end{equation}
Here $\Omega \subset \mathbb{R}^n$ is a bounded smooth
  domain with
exterior unit normal $\nu$ at $\partial \Omega$, and $f$ is a
sufficiently smooth datum. If $n=2$, the unknown $u$ models the
vertical deflection of a horizontally clamped thin elastic plate
{from} the horizontal equilibrium shape under the vertical load $f$.
The boundary conditions $u_{\vert _{\partial \Omega }}=\frac
\partial {\partial \nu}u_{\vert _{\partial \Omega }}=0$ are called  \emph{Dirichlet boundary conditions} 
and are natural in mechanics to model the horizontal clamping: More precisely, the condition  
$u_{\vert _{\partial \Omega }}=0$ models the location of the clamping and the condition $\frac\partial {\partial \nu}u_{\vert _{\partial \Omega }}=0$ models the fact that the 
plate is clamped into some matter and is not able to rotate freely. We refer to the memoir written 
by Hadamard \cite{Hadamard1} for further considerations on this question. 
Throughout the present paper, always these boundary conditions will be considered.

\medskip We shall discuss comparison principles or positivity preserving properties 
for the biharmonic operator. We say that the clamped plate boundary value problem enjoys a
\emph{positivity preserving property } in $\Omega$ if the following assertion holds:
\begin{equation*}
\hbox{for all }u\in C^4(\overline{\Omega})\hbox{ and }f\in C^0(\overline{\Omega})\hbox{ satisfying }\eqref{cpe}, \hbox{ one has that }
\end{equation*}
\begin{equation*}
 \left\{\; f\geq 0\;\Rightarrow \; u\geq 0\; \right\}.
\end{equation*}
This definition is the natural extension of the ``positivity preserving property'' for the harmonic 
operator, i.e. a comparison principle for a second-order operator. 
While the ``positivity preserving property'' is well understood for second-order operators, it is much more involved for fourth-order ones.

\medskip  The positivity preserving property is closely related to the sign of the Green's function. More precisely, let $H_{\Omega}=H_{\Omega,\Delta^2}$ be the singular Green's function for the operator $\Delta^2$ in $\Omega$ under Dirichlet boundary conditions. 
Then, for any reasonable datum $f:\Omega\to\mathbb{R}$ the solution
$u:\overline{\Omega}\to\mathbb{R}$ of the clamped plate boundary value problem
(\ref{cpe}) is given by
$$
u(x)=\int_{\Omega} H_{\Omega}(x,y) f(y)\, dy.
$$
In particular, the clamped plate boundary value problem enjoys a
\emph{positivity preserving property } in $\Omega$ iff $H_\Omega(x,y)\geq 0$ for all $x,y\in\Omega$, $x\neq y$. 

\medskip  It is important to remark  that positivity issues are related to the 
specific kind of prescribed boundary conditions. 
More precisely, if one chooses  Navier boundary condition (that is $u=\Delta u=0$ on $\partial\Omega$), then a twofold application of the second order comparison principle immediately yields 
a positivity preserving property. This simple situation is misleading in several respects: 
As we will explain below, counterexamples show that the situation is much more intricate for other
 boundary conditions. For Dirichlet boundary conditions which we consider here, this iterative
trick fails completely. Moreover, even under Navier conditions, there is in general no 
positivity preserving property for perturbations of the biharmonic operator,
 see \cite{Schroeder,KawohlSweers}, cf. also the general approach in \cite{CoffmanGrover}.

\medskip The first example of a positive Green's function was  given by Boggio  \cite{Boggio}
by means of a beautiful explicit formula
for balls in $\mathbb{R}^n$,
even for the Dirichlet problem for \emph{polyharmonic operators}.  Boggio \cite{Boggio1} (1901) and Hadamard \cite{Hadamard1} (1908)
conjectured that in arbitrary reasonable
(two dimensional) domains $\Omega$, the positivity preserving
property should hold true.  In 1909
Hadamard \cite {Hadamard2} already knew, that the positivity conjecture
is false in annuli with small inner radius. However, there was still some hope to prove positivity for convex domains.

\medskip  Starting about 40 years later, numerous counterexamples disproved the
Boggio-Hadamard conjecture, see e.g.  \cite{Duffin0,Garabedian,ShapiroTegmark}. In particular, Coffman and Duffin \cite{CoffmanDuffin} proved that in any two dimensional domain
containing a right angle, the positivity preserving property does not hold. 
Even for smooth convex domains, the issue is quite intricate: Garabedian \cite{Garabedian} (see also Shapiro-Tegmark \cite{ShapiroTegmark}
and Hedenmalm-Jakobsson-Shimorin\cite{Hedenmalm}) proved that for mildly eccentric
ellipses, the prositivity preserving property does not hold true. 
On the other hand, according to \cite{GrunauSweers1}, one has positivity in ellipses,
which are close enough to a ball.
For a more extensive survey and further references we also refer to
\cite{GrunauSweers1}.

\medskip  Therefore, the positivity preserving property does not hold true in general, even for arbitrarily smooth uniformly convex domains. Hence, it is important to understand the lack of ``positivity preserving property'' with the help of the Green's function. One may  ask the following questions:
\begin{enumerate}
 \item 
Is positivity preserving in any bounded smooth domain possibly
``almost true'' in the sense that the negative part $H_{\Omega}^-(x,y):=
\min\{H_{\Omega}(x,y),0\}$ of the biharmonic Green's function
under Dirichlet boundary conditions is ``small relatively'' to the 
singular positive part $H_{\Omega}^+(x,y)$?
\item 
Are there are at least families of domains, different from balls,
where the biharmonic Green's functions
under Dirichlet boundary conditions  are
positive?
\end{enumerate}

Question (1) is motivated by reactions of physicists and engineers
on the mathematical results concerning positivity preserving and
sign change. They may be summarised as follows: ``For a clamped
plate without corners, we do not expect downwards deflections
if the force is pushing upwards. If such a phenomenon can be 
mathematically observed we think that perhaps the model is not
perfectly suitable or that negativity is so small that it cannot be
observed in reality.'' Also numerical experiments give support
to the second hypothesis, which is subject of our first main result.

\medskip  The general behaviour of the Green's functions is modeled on the behaviour of the 
singular fundamental solution on $\mathbb{R}^n$. On the whole space, we have that (letting  $e_n$ be the $n$-dimensional volume of $B_1(0)\subset\mathbb{R}^n$)

$$H_{\mathbb{R}^n}(x,y)=\left\{\begin{array}{ll}
\displaystyle\frac{1}{2(n-4)(n-2)n e_n}|x-y|^{4-n}, & \hbox{ when }n\geq 5,\\[4mm]
\displaystyle\frac{1}{16e_4}\log\frac{1}{|x-y|} ,& \hbox{ when }n=4,\\[4mm]
\displaystyle-\frac{1}{8\pi}|x-y|, & \hbox{ when }n=3,\\
\end{array}\right.$$
for all $x,y\in\mathbb{R}^n$, $x\neq y$. If $n\ge 5$, this fundamental solution can even be interpreted
as a Green's function in $\mathbb{R}^n$, where
the Dirichlet boundary conditions  at infinity are understood
as a suitable decay. In the general framework of a bounded smooth domain of $\mathbb{R}^n$, Krasovski\u{\i} \cite{Krasovskij1,Krasovskij2} proved that there exists a constant $C(\Omega)$ such that
$$
|H_\Omega (x,y)| \le C(\Omega) \left\{\begin{array}{ll}
|x-y|^{4-n},\quad & \mbox{\ if\ }n>4,\\
\left( 1+|\log|x-y|\,| \right),& \mbox{\ if\ }n=4,\\
1,& \mbox{\ if\ }n<4;\\
\end{array}\right.
$$
for all $x,y\in \Omega$, $x\neq y$.  These estimates give a uniform bound
for the singular behaviour for Green's functions but do not consider their
boundary behaviour. The latter was done by  Dall'Acqua and Sweers \cite{DallacquaSweersJDE}
by means of integrating Krasovski\u{\i}'s estimates for $H$ and its derivatives.:
\begin{equation}\label{twosidedestimate}
 |H_{\Omega} (x,y)| \le\left\{ \begin{array}{ll}
\displaystyle C \cdot |x-y|^{4-n} \min\left\{
1,\frac{d(x)^2d(y)^2}{|x-y|^4}
\right\},
&\mbox{if\ } n>4,
\\[4mm]
\displaystyle C \cdot \log\left(1+\frac{d(x)^2d(y)^2}{|x-y|^4} \right),
&\mbox{if\ } n=4,
\\[4mm]
\displaystyle C \cdot
d(x)^{2-n/2} d(y)^{2-n/2}  \min\left\{
1,\frac{d(x)^{n/2} d(y)^{n/2} }{|x-y|^n}
\right\},
&\mbox{if\ } n<4.
\end{array}
\right.
\end{equation}
Here, $d(x)=\operatorname{d}(x,\partial\Omega)$ and $C=C(\Omega)>0$
denotes a constant.

As far as the positive part $H_{\Omega} ^+$ is concerned, these
estimates cannot be improved, see e.g. \cite{GrunauSweers2}.
However, they do not distinguish between the positive and the
negative part of the Green's function. A distinction between
$H_{\Omega}^+$ and $H_{\Omega}^-$ and showing
that pairs of points of negativity cannot approach each other is the  subject
of our first main result, which was announced in \cite{GrunauRobert}.
\begin{theorem}\label{smallnegativepart}
Let  $\Omega\subset \mathbb{R}^n$ $(n\ge 3)$ be a bounded $C^{4,\alpha}$-smooth
domain. We denote by $H_{\Omega}$ the biharmonic Green's function in 
$\Omega$ under Dirichlet boundary conditions. 
Then, there exists a constant $\delta=\delta(\Omega)>0$ such that for any two points
$x,y\in\Omega$, $x\not=y$,
$$
H_{\Omega}(x,y)\le 0 \mbox{\ implies that\ } |x-y|\ge \delta.
$$
Consequently, there exists a constant $C=C(\Omega)>0$ such that for any two points
$x,y\in\Omega$, $x\not=y$, we have that
\begin{equation}\label{boundfrombelow0}
H_{\Omega}(x,y)\geq -C(\Omega),
\end{equation}
i.e. the negative part $H_\Omega^-$ of the Green's function is bounded.
Moreover, if $\Omega$ is smooth enough for (\ref{twosidedestimate})
to hold true, the estimate (\ref{boundfrombelow0}) from below can be refined:
\begin{equation}\label{boundfrombelow}
 H_{\Omega}(x,y)\ge -C(\Omega) \, d(x)^2\, d(y)^2.
\end{equation}
\end{theorem}
In other words: Around the pole, biharmonic Green's functions are always
positive, if $n\ge 3$. And this behaviour is uniform, even if the pole approaches the boundary.

The proof of Theorem~\ref{smallnegativepart} indicates that one may not expect the full 
result to hold true also for $n=2$. The bound \eqref{boundfrombelow},
however, was proved for the case $n=2$ and sufficiently smooth domains
by Dall'Acqua, Meister
and Sweers \cite{DallacquaMeisterSweers}. Even for $n=2,3$, where the
Green's function is bounded, (\ref{boundfrombelow}) is a strong statement
because in the case, where $x$ or $y$ is closer to the boundary than
they are to each other, (\ref{twosidedestimate}) would only give
$
H_{\Omega}\ge -C \frac{d(x)^2d(y)^2}{|x-y|^n}.
$
In this sense, (\ref{boundfrombelow}) gains a factor of order $|x-y|^n$.

In his counterexample to positivity mentioned above,
Garabedian \cite{Garabedian} found in a mildly eccentric ellipse $\Omega$  opposite boundary
points $x_0,y_0\in\partial\Omega$ with $\Delta_x\Delta_y H_{\Omega}(x_0,y_0)<0$.
This shows that qualitatively, the estimate (\ref{boundfrombelow}) is sharp.

\medskip Another consequence of Theorem \ref{smallnegativepart} is that one has a 
strong control of the negative part of solutions $u$ of the clamped plate
boundary value problem (\ref{cpe}) with nonnegative datum $f$
irrespective of the space dimension:
$$
f\ge 0\quad  \Rightarrow \quad
\| u^-\|_{L^\infty(\Omega)}\le C(\Omega)
\| f\| _{L^1 (\Omega)}.
$$
This  estimate should be compared with the estimate for the full function: 
It follows from general elliptic theory \cite{ADN} that for all $p>\frac{n}{4}$, there exists $C(p,\Omega)>0$ such that for solutions $u$ of the clamped plate
boundary value problem (\ref{cpe}) with datum $f$, we have that
$$\Vert u\Vert_{L^\infty(\Omega)}\leq C(p,\Omega)\Vert f\Vert _{L^p(\Omega)}.$$
Consequently, on the one hand, there might be a negative part for $u$, even if $f\ge 0$.
But on the other hand, this negative part enjoys a nice strong control by a very weak norm. 
Therefore, although one has a lack of positivity in general,  in this sense it is ``small''.

\bigskip
Let us now turn to Question (2): Are there  at least families of domains, 
different from balls, where the biharmonic Green's functions
under Dirichlet boundary conditions  are
positive?
For two dimensional domains, this question was addressed in \cite{GrunauSweers1}.
There, it was shown that in domains $\Omega\subset\mathbb{R}^2$ being
sufficiently close in $C^4$-sense
to the (unit) disk $B\subset\mathbb{R}^2$, the biharmonic
Green's  function (under Dirichlet boundary conditions) is positive.
Recently, Sassone \cite{Sassone} could relax the assumption on the domains
to be close to $B$ in $C^{2,\alpha}$-sense. The authors could take advantage
of conformal maps and the Riemann mapping theorem, pulling back
the clamped plate boundary value problem to Dirichlet problems in the
unit disk with the biharmonic operator as principal part and only
with (small) \emph{lower order perturbations}. The latter were
treated in $B\subset\mathbb{R}^n$ ($n$ arbitrary) in \cite{GrunauSweers2}.
The methods of \cite{GrunauSweers1}, however, do not carry over
to higher dimensions due to a lack of sufficiently many conformal maps.
So, the question, whether the positivity of the biharmonic Green's  function
in the unit ball $B\subset\mathbb{R}^n$ is stable under
\emph{domain perturbations}, was left open.

This question  is addressed  in our next result. Assuming
$n>2$, we show that in domains $\Omega\subset\mathbb{R}^n$, which are
sufficiently close to the unit ball in a suitable
$C^{4,\alpha}$-sense, the biharmonic Green's  function under Dirichlet boundary conditions
is indeed positive.
More precisely, we prove the following theorem, where $Id$ denotes the 
identity map:

\begin{theorem} \label{theorem1}
Let $B$ be a unit ball of $\rn$, $n\ge 3$. Then, there exists $\varepsilon_0=\varepsilon_0(n)>0$
such that the following holds true:

We assume that $\Omega\subset\mathbb{R}^n$ is a
$C^{4,\alpha}$-smooth  domain which is $\varepsilon_0$-close to the
ball $B$ in the $C^{4,\alpha}$-sense, i.e.:
\begin{quote}
There exists a surjective $C^{4,\alpha}$-map $\psi:\overline{B} \to\overline{\Omega}$
such that \\
$\| Id -\psi\|_{C^{4,\alpha}\left( \overline{B}\right)}\le\varepsilon_0$.
\end{quote}
Then,  the Green's  function $H_{\Omega}$ for $\Delta^2$
in $\Omega$ under  Dirichlet boundary conditions is strictly positive:
$$
\forall x,y \in \Omega,\quad x\not=y:\qquad
H_{\Omega} (x,y) >0.
$$
\end{theorem}

Assuming $\varepsilon_0$ small enough, this notion of closeness
implies  that there is a fixed neighbourhood
$U$ of $\overline{B}$, $C^{4,\alpha}$-smooth injective extensions
$\psi:U\to\mathbb{R}^n$, $\| Id
-\psi\|_{C^{4,\alpha}\left( U\right)}=O(\varepsilon)$
and $C^{4,\alpha}$-smooth  inverse maps
$\phi=\psi^{-1}:\psi(U)\to U$ such that
$\psi\left( \overline{B}\right) =\overline{\Omega}$, $\psi\left(
{B}\right)={\Omega}$.

\medskip  For $n=2$,  a direct and explicit proof based on perturbation series, Green's  function
estimates and conformal maps was given
in \cite{GrunauSweers1,GrunauSweers2}. This means that there,
in principle, $\varepsilon_0$ may be calculated explicitly.
Moreover, in the case $n=2$, closeness has to be assumed only in a
weaker norm, see \cite{Sassone}.

\medskip  Here, the proof is somehow more indirect since a number of proofs
by contradiction is involved so that it will be  impossible  to calculate
$\varepsilon_0$ {from} our proofs. Furthermore, we have to make
extensive use of general elliptic theory as provided by Agmon,
Douglis and Nirenberg \cite{ADN}. We emphasize that
Theorem~\ref{theorem1} is by no means just a continuous dependence
on data result, since the involved Green's ``functions'' are not
simply functions but families of functions depending on the
position of the singularity.

\medskip  The problem consists in gaining
uniformity with respect to the position of the singularity: When the singularity is in the interior, 
it is possible to use the local positivity results of Grunau-Sweers \cite{GrunauSweers3}. But when the singularity approaches the boundary, the situation becomes more intricate and we develop a 
new technique for this problem. These remarks
apply also to Theorem~\ref{smallnegativepart} which is proved by means
of the same methods.

\medskip  It remains as an interesting question to find out an optimal notion
of closeness for a result like Theorem~\ref{theorem1} to hold true.
One might expect that   like in the two-dimensional 
case (see\cite{Sassone}) $C^{2,\alpha}$-closeness could
suffice: Does the
boundary curvature govern the positivity behaviour of
biharmonic Green's functions? However, such a result
would require new ideas and techniques; its possible
proof would certainly be much more technical
than ours.  Sassone's approach is strictly
two-dimensional because of his use of conformal maps.

\medskip  Our methods and techniques  are general and 
our results can be extended to more general
fourth order ``positive definite self adjoint'' elliptic operators
under Dirichlet boundary conditions, where the principal part
is a square of second order elliptic operators, and also to 
similar elliptic operators of higher order $2m$ 
in dimensions $n\ge 2m-1$ with reasonable boundary conditions of the type discussed in Agmon-Douglis-Nirenberg \cite{ADN}.  



\section{A more general result}\label{section0}

In order to prove Theorem~\ref{theorem1},
below in  Theorem~\ref{theorem2} we describe the possible situations
how transition from positivity to sign change may occur within a smooth
family of domains. It is then easy to see that none of these situations occurs
in the (unit) ball in $\mathbb{R}^n$, $n>2$. Moreover, a special case
of Theorem~\ref{theorem2} will directly yield the proof of
 Theorem~\ref{smallnegativepart}.

To provide a more flexible
result in Theorem~\ref{theorem2}, we will also include lower order perturbations.
The formulation is somehow technical and requires in particular the notion
of smooth domain perturbations, which we make precise in the following definition.

\begin{definition}\label{definition1}
Let $\Omega$, $(\Omega_k)_{k\in\nn}$ be  domains of $\rn$. We say
that $(\Omega_k)_{k\in\nn}$ is a \emph{$C^{4,\alpha}$-smooth
perturbation} of the bounded $C^{4,\alpha}$-smooth  domain $\Omega$, and
we write
$$\lim_{k\to +\infty}\Omega_k=\Omega$$
if the following facts are satisfied:
\begin{itemize}
\item[(i)]  There exist $N\in\nn$,  $p_1,\dots,p_N\in\partial\Omega$,
$\delta>0$ and open subsets $\omega\subset\subset\Omega$,
$\omega\subset\subset\omega_0\subset\subset\Omega_k$ such that
$$\Omega\subset \omega\cup\bigcup_{i=1}^NB_\delta(p_i);
\qquad
\Omega_k\subset \omega\cup\bigcup_{i=1}^NB_\delta(p_i);
$$
\item[(ii)] for any $i\in\{1,\dots,N\}$, there exists  an open subset of
$U_i\subset \rn$ such that
$0\in U_i$, and  a $C^{4,\alpha}$-smooth diffeomorphism
$\Phi_i: U_i\to B_{2\delta}(p_i)$
such that $\Phi_i(0)=p_i$ and
$$
\Phi_i(U_i\cap\{x_1<0\})=\Phi_i(U_i)\cap\Omega\; ,\;
 \Phi_i(U_i\cap\{x_1=0\})=\Phi_i(U_i)\cap\partial\Omega;
$$
\item[(iii)] for any $i\in\{1,\dots,N\}$ and $k\in\nn$, there
exists $\phiki: U_i\to B_{2\delta}(p_i)$ such that $\phiki(U_i)$
is an open subset, $B_{3\delta/2}(p_i)\subset \phiki(U_i)$ and
$\phiki : U_i\to \phiki(U_i)$ is a $C^{4,\alpha}$-smooth
diffeomorphism and
$$\phiki(U_i\cap\{x_1<0\})=\phiki(U_i)\cap\Omega_k\; ,\;
 \phiki(U_i\cap\{x_1=0\})=\phiki(U_i)\cap\partial\Omega_k;
$$
\item[(iv)] for any $i\in\{1,\dots,N\}$, $\lim_{k\to +\infty}\phiki=\Phi_i$ in
$C^{4,\alpha}_{loc} (U_i)$.
\end{itemize}
\end{definition}

This definition implies that we have a  well defined smooth exterior normal
vector field so that $\Omega$,  $\Omega_k$, $\partial \Omega$
and $\partial \Omega_k$ carry a canonical orientation.
In what follows, the local charts will be chosen such that this orientation is
observed, i.e. such that $\hbox{Jac }\Phi_i \circ\Phi_j^{-1}>0$,
$\hbox{Jac }\Phi_{k,i} \circ\Phi_{k,j}^{-1}>0$.

\medskip This definition covers in particular the following
more special situation of smooth domain perturbation,
which we make use of in proving Theorem~\ref{theorem1}:
Let a sequence of mappings $(\psi_k)_{k\in\nn}$ be
such that there exists  an open subset of $U\subset \rn$ and $\psi_k:U\to\rn$ for all $k\in\nn$.
We assume that $\lim_{k\to +\infty}\psi_k=Id$ in $C^{4,\alpha}_{loc}(U)$.
Let $\Omega\subset\subset U$ be a $C^{4,\alpha}$-smooth bounded  subset of $\rn$
and let $\Omega_k:=\psi_k(\Omega)$ for all $k\in\nn$. Then the sequence
$(\Omega_k)_{k\in\nn}$ is a smooth perturbation of $\Omega$.

Employing  the notion of smooth domain perturbation we are now able to
formulate our key result (where Theorems~\ref{smallnegativepart} and \ref{theorem1} 
are a  consequence of as it is explained at the end of Section~\ref{section3}):

\begin{theorem} \label{theorem2}
Let  $n\ge 3$, and $(\Omega_k)_{k\in\nn}$ be a
$C^{4,\alpha}$-smooth perturbation of the bounded
$C^{4,\alpha}$-smooth  domain $\Omega$ in the sense of
Definition~\ref{definition1}. We consider a sequence
$(a_k)_{k\in\nn}\in C^{0,\alpha}(U_0)$, where
$\Omega\subset\subset U_0$ and assume that there exists
$a_\infty\in C^{0,\alpha}(U_0)$ such that
$$\lim_{k\to +\infty} a_k=a_\infty\hbox{ in }C^{0,\alpha}_{loc}(U_0).$$
We assume further that there exists $\lambda>0$ such that
\begin{equation}\label{uniformcoercive}
\int_{\Omega_k}\left((\Delta\varphi)^2+a_k\varphi^2\right)\, dx
\geq \lambda\int_{\Omega_k}\varphi^2\, dx
\end{equation}
for all $\varphi\in C^\infty_c(\Omega_k)$ and all $k\in\nn$. Let $G_k$ be the Green's  function of
$\Delta^2+a_k$ on $\Omega_k$, and $G$ be the
Green's  function of $\Delta^2+a_\infty$ on $\Omega$, all with Dirichlet boundary conditions.

Finally, we suppose
 that there exist two sequences $(x_k)_{k\in\nn},(y_k)_{k\in\nn}$ such
that $x_k,y_k\in\Omega_k$ and
$$G_k(x_k,y_k)=0\hbox{ for all }k\in\nn.$$
Up to a subsequence, let $x_\infty:=\lim_{k\to +\infty}x_k$ and $y_\infty:=\lim_{k\to +\infty}y_k$.
Then $x_\infty,y_\infty\in\overline{\Omega}$, $x_\infty\neq y_\infty$ and we are in one of the
following situations:
\begin{itemize}
\item[(i)] if $x_\infty,y_\infty\in\Omega$, then $G(x_\infty,y_\infty)=0$;
\item[(ii)]
if $x_\infty\in\Omega$ and $y_\infty\in\partial\Omega$, then $\Delta_y G(x_\infty,y_\infty)=0$;
\item[(iii)]
if $x_\infty\in\partial\Omega$ and $y_\infty\in\Omega$, then $\Delta_x G(x_\infty,y_\infty)=0$;
\item[(iv)]
if $x_\infty,y_\infty\in\partial\Omega$, then $\Delta_x\Delta_y G(x_\infty,y_\infty)=0.$
\end{itemize}
\end{theorem}

 In the above statement,  $\Delta_x G$ denotes the
Laplacian with respect to the first variables, and $\Delta_y G$ denotes the Laplacian with respect to
the second variables.
The uniform coercivity assumption \eqref{uniformcoercive} is e.g. implied
by a sign condition like $a_k \ge 0$ and uniform smoothness of the domains 
to have a uniform Poincar\'{e}-Friedrichs inequality.

A result like Theorem~\ref{theorem2} seems to be not known if $n=2$ and we doubt
whether it is correct there.  At least, our method cannot be extended to
this case. The proof of Lemma~\ref{lemma3.3} in dimensions
$n=3,4$ uses a Nehari-type  local positivity result \cite{GrunauSweers3,Nehari} which is not available
for $n=2$.

More general lower order ``self adjoint'' perturbations of the biharmonic operator
may be covered by precisely the same techniques. However, here we prefer to
stick to a relatively simple situation in order to avoid too many technical details.

In the one dimensional context (clamped bars), related and quite concrete results
were obtained by Schr\"oder \cite{SchroederMZ,Schroeder1,Schroeder}.

In order to gain a better feeling for the statement of Theorem~\ref{theorem2}
one should think of deforming the ball, where we know that positivity
preserving holds true, smoothly into a domain where the biharmonic Green's
function is sign changing (e.g. a long thin ellipsoid). There is a ``last''
domain where one still has a nonnegative Green's function. Our result
describes the possible degeneracies of this positivity via which
sign change occurs beyond this ``last positivity-domain''.
The key statement is that $x_\infty\neq y_\infty$ so that
it is impossible that negativity appears through the singularity:
Around the singularity, our Green's functions are always positive.
The most difficult part is to prove this also arbitrarily close 
to the boundary.

Alternatively one may think of the Green's function for $\Delta^2+\lambda$
in a ball for $\lambda \to \infty$ where again, initially one has positivity
while sign changes occur for $\lambda \to \infty$ (see e.g. \cite{Bernis,CoffmanGrover}).
We think that most likely the transition from positivity to sign change
will occur via alternative (iv) of Theorem~\ref{theorem2}.

Throughout the paper we assume that
$$
n\ge 3.
$$

A first essential step in proving Theorem~\ref{theorem2} consists in providing uniform
bounds (in $k$)
for the Green's  functions like
\begin{equation}\label{0.0}
 |G_k (x,y)|\le C \cdot
\left\{\begin{array}{ll}
 |x-y|^{4-n},\quad&\mbox{\ if\ } n>4,\\
\left(1+\left|\log |x-y|\right|\right),\quad&\mbox{\ if\ } n=4,\\
1,\quad&\mbox{\ if\ } n=3.
\end{array}\right.
\end{equation}
Moreover, if $n=3,4$, the somehow irregular estimates for $G_k$ require
to focus first on the  gradients, where estimates like
\begin{equation}\label{0.1}
 |\nabla G_k (x,y)|\le C \cdot
\left\{\begin{array}{ll}
\left| x-y\right|^{-1},\quad&\mbox{\ if\ } n=4,\\
1,\quad&\mbox{\ if\ } n=3,
\end{array}\right.
\end{equation}
are available, which are compatible with the scaling arguments performed below.
In this respect, the proof is more difficult in dimensions $n=3$ and in particular
$n=4$.
Estimates (\ref{0.0}) and (\ref{0.1}) are
due to Krasovski\u{\i} \cite{Krasovskij2}, provided the family $(\Omega_k)$ 
is assumed to be uniformly $C^{11}$-smooth. This assumption is due
to the great generality of the boundary value problems considered
by Krasovski\u{\i}. We prove in Theorem~\ref{theorem3} that  in our special situation,
(\ref{0.0}) and (\ref{0.1}) hold true under our uniform
$C^{4,\alpha}$-smoothness assumptions. 
 Preliminary properties of the Green's  functions are shown in Section~\ref{section1},
while Section~\ref{section2} is devoted to convergence properties of families
of Green's  functions in $(\Omega_k)_{k\in\mathbb{N}}$. The proofs of
Theorems~\ref{theorem2} and, as a consequence, of Theorems~\ref{smallnegativepart}
and \ref{theorem1} are finally given in Section~\ref{section3}.

{\bf Notation.} Straightening the boundary requires to work in
$\mathbb{R}^n_-:=\{x\in \mathbb{R}^n:x_1<0\}$, where
we write $\mathbb{R}^n\ni x=(x_1,\bar{x}).$
$e_n$ denotes the $n$-dimensional volume of $B_1(0)\subset\mathbb{R}^n$.
$C^{\infty}_c (\Omega)$ denotes the space of arbitrarily smooth
functions with compact support in $\Omega$ and ${\mathcal D}'(\Omega)$
its dual, i.e. the space of distributions on $\Omega$.


\section{The Green's  function $G$ for the perturbed biharmonic operator}
\label{section1}

In the first part of this section, we consider a fixed operator $\Delta^2+a$ in
a fixed smooth domain and construct and investigate the corresponding Green's
function.

\begin{proposition} \label{proposition1}
Let $\Omega\subset\mathbb{R}^n$ be a bounded $C^{4,\alpha}$-smooth
 domain and
 $a\in C^{0,\alpha} \left( \overline{\Omega}\right)$. We assume that $\Delta^2+a$
is coercive. Then, for every $x\in \Omega$, there exists a unique Green's  function
$G_x\in L^1(\Omega)\cap C^{4,\alpha}  \left( \overline{\Omega}\setminus\{x\}\right)$ such that
$G_x|_{\partial \Omega}=\frac{\partial G_x}{\partial \nu}|_{\partial \Omega}=0$ and that for all
$\varphi\in C^4\left( \overline{\Omega}\right)$ with $\varphi|_{\partial \Omega}
=\frac{\partial \varphi}{\partial \nu}|_{\partial \Omega}=0$ one has the following
representation formula:
\begin{equation}\label{1.000}
\varphi(x)=\int_{\Omega} G_x(y) \left( \Delta^2 \varphi(y)+ a(y)\varphi(y)\right)\, dy.
\end{equation}
If $R>0$ is such that $\Omega\subset B_R(0)$ and $\lambda>0$ such that
$$\forall \varphi\in W^{2,2}_0 (\Omega):\
\int_{\Omega} \left( (\Delta\varphi)^2+a\varphi^2\right)\, dy \ge
         \lambda\int_{\Omega}\varphi^2\, dy
$$
then, the following estimate for the Green's  function holds true:
\begin{eqnarray}
|G_x (y)| &\le & C(\lambda,R,n,\|a\|_{C^{0,\alpha}},\Omega)\label{1.00}\\
&&\cdot \left\{ \begin{array}{ll}
\left( |x-y|^{4-n} + \max\{d(x,\partial\Omega),
d(y,\partial\Omega)\}^{4-n} \right),\quad&\mbox{\ if\ } n>4,\\
1+\left| \log|x-y| \right|+\left|\log\left(\max\{d(x,\partial\Omega),
d(y,\partial\Omega)\}  \right)\right|,&\mbox{\ if\ } n=4,\\
1,&\mbox{\ if\ } n=3.
\end{array}\right.\nonumber
\end{eqnarray}
If $n=3,4$, we further prove the following gradient estimates:
\begin{eqnarray}
|\nabla_{(x,y)} G_x (y)| &\le & C(\lambda,R,n,\|a\|_{C^{0,\alpha}},\Omega)
\label{1.01}\\
&&\cdot \left\{ \begin{array}{ll}
\left( |x-y|^{-1} + \max\{d(x,\partial\Omega),
d(y,\partial\Omega)\}^{-1} \right),\quad&\mbox{\ if\ } n=4,\\
1,&\mbox{\ if\ } n=3.
\end{array}\right.\nonumber
\end{eqnarray}
The dependence of  the constants
$C$ on $\Omega$ is explicit via the $C^{4,\alpha}$-properties of
$\partial \Omega$.
\end{proposition}

\begin{proof} We first prove extensively the generic case $n>4$. At the end
we comment on the changes and additional arguments which have to be made,
if $n=4$ or $n=3$.

\bigskip\noindent
{\it Case $n>4$.}
We introduce the {fundamental solution} $\Gamma_0$ of the biharmonic operator
\begin{equation}\label{1.0}
\Gamma_0 (x,y):=\frac{1}{2(n-4)(n-2)ne_n}|x-y|^{4-n}
\end{equation}
so that
$\Gamma_0\in C^{\infty}
\left( \overline{\Omega}\times\overline{\Omega} \right)\setminus \{ (x,y): x=y\}$.
We define recursively for $j\ge 0$
$$
\Gamma_{j+1} (x,y):=-\int_{\Omega} \Gamma_{j} (x,z)a(z)\Gamma_0(z,y)\, dz
$$
and have that
$\Gamma_j \in C^{4,\alpha}\left( \overline{\Omega}\times\overline{\Omega}
\setminus \{ (x,y): x=y\}\right)$
is well defined and, according to a Lemma of Giraud \cite{Giraud}, that for $j\ge 1$
\begin{equation}\label{1.1}
\left| \Gamma_{j} (x,y)\right| \le\left\{\begin{array}{ll}
C_j |x-y|^{4(j+1)-n},\quad&\mbox{\ if\ } (j+1)<\frac{n}{4},\\
C_j\left( 1+\left| \log |x-y|  \right| \right),\quad&\mbox{\ if\ } (j+1)=\frac{n}{4},\\
C_j,\quad&\mbox{\ if\ } (j+1)>\frac{n}{4}.
\end{array}\right.
\end{equation}
Here, $C_j=C_j\left(n,R,\| a\|_{\infty}\right)$, where $R>0$ is chosen such that
$\Omega\subset B_R(0)$.
We fix some $\ell>\frac{n}{4}$, $x\in \Omega$ and for
$u_x\in C^{4,\alpha} \left( \overline{\Omega}\right)$
to be suitably determined below, we put
\begin{equation}\label{1.2}
G_x (y) := \Gamma_0(x,y)+\sum^{\ell}_{j=1} \Gamma_j(x,y)+u_x(y).
\end{equation}
One should observe that $\sum^{\infty}_{j=0} \Gamma_j$ is the Neumann-series
for the fundamental solution for the perturbed differential operator.
We have that $G_x\in C^{4,\alpha} \left( \overline{\Omega}\setminus \{ x\}\right)$. 
In order that $G_x$ becomes indeed a Green's  function for the
Dirichlet  problem for
$\Delta^2+a$, i.e. that indeed formula (\ref{1.000}) is satisfied,
we need $u_x$ to be a solution of the following Dirichlet problem
\begin{equation}\label{1.3}
\left\{ \begin{array}{ll}
\Delta^2 u_x(y)+a(y) u_x(y)=-a(y) \Gamma_{\ell} (x,y)\quad&\mbox{\ in\ }\Omega\\
u_x(y)=-\Gamma_0 (x,y) -\sum^{\ell}_{j=1} \Gamma_j (x,y)\quad&\mbox{\ for\ }
                        y\in \partial\Omega,\\
\frac{\partial}{\partial \nu}u_x(y)=
  -\frac{\partial}{\partial \nu_y}\Gamma_0 (x,y)
-\sum^{\ell}_{j=1} \frac{\partial}{\partial \nu_y}\Gamma_j (x,y)\quad&\mbox{\ for\ }
                        y\in \partial\Omega.
\end{array}\right.
\end{equation}
Since $\ell>\frac{n}{4}$, the right hand side $-a\cdot \Gamma_{\ell} (x,\, .\, )$
is H\"older continuous with H\"older norm bounded by
a constant $C(n,R,\|a\|_{C^{0,\alpha}})$.
The $C^{1,\alpha}$-norm of the datum for $u_x|_{\partial \Omega}$
and the $C^{0,\alpha}$-norm of the datum for $\frac{\partial}{\partial \nu}u_x|_{\partial \Omega}$
are bounded by
 $C(n,R,\partial\Omega)   d(x,\partial \Omega)^{3-n-\alpha}$. The dependence of
the constant $C$ on $\partial\Omega$ is in principle constructive and explicit  via its curvature
properties and their derivatives.
According to $C^{1,\alpha}$--estimates for boundary value problems in variational
form like (\ref{1.3}) -- see \cite[Thm. 9.3]{ADN} -- we see that
\begin{equation}\label{1.4}
\| u_x\|_{C^{1,\alpha}\left( \overline{\Omega} \right)}
\le C(n,R,\lambda,\|a\|_{C^{0,\alpha}},\partial\Omega)  d(x,\partial \Omega)^{3-n-\alpha}.
\end{equation}
One should observe that the differential operators are uniformly coercive,
so that no $u_x$-terms need to appear on the right-hand-side.

As long as $d(y,\partial \Omega)\le d(x,\partial \Omega)$, (\ref{1.4}) shows that
$$
| u_x (y) | \le  C(C_0,n,R,\lambda,\|a\|_{C^{0,\alpha}},\partial\Omega)
d(x,\partial \Omega)^{4-n}
$$
and hence
\begin{equation}\label{1.5}
|G_x(y) | \le  C(C_0,n,R,\lambda,\|a\|_{C^{0,\alpha}},\partial\Omega)
\left( |x-y|^{4-n} +d(x,\partial \Omega)^{4-n} \right).
\end{equation}
If $d(y,\partial \Omega) > d(x,\partial \Omega)$ we conclude from (\ref{1.5})
by exploiting the symmetry of the Green's  function:
\begin{equation}\label{1.6}
|G_x(y) | =|G_y(x)| \le  C(C_0,n,R,\lambda,\|a\|_{C^{0,\alpha}},\partial\Omega)
\left( |x-y|^{4-n} +d(y,\partial \Omega)^{4-n} \right).
\end{equation}
Combining (\ref{1.5}) and (\ref{1.6}) yields (\ref{1.00}) for $n>4$.

\bigskip\noindent
{\it Case $n=4$.}
Here the fundamental solution we start with is
\begin{equation}\label{1.61}
\Gamma_0 (x,y):=-\frac{1}{16 e_4}\log |x-y|.
\end{equation}
We proceed with the iterated kernels $\Gamma_j$. In view of the
mild singularity of $\Gamma_0$, however, it is sufficient to choose
$\ell=1$. As above we find that
\begin{equation}\label{1.62}
\| u_x\|_{C^{1,\alpha}\left( \overline{\Omega} \right)}
\le C(n,R,\lambda,\|a\|_{C^{0,\alpha}},\partial\Omega)  d(x,\partial \Omega)^{-1-\alpha}.
\end{equation}
As long as $d(y,\partial \Omega)\le d(x,\partial \Omega)$, (\ref{1.62}) shows that
\begin{equation}\label{1.63}
| \nabla_y G_x (y) | \le  C(C_0,n,R,\lambda,\|a\|_{C^{0,\alpha}},\partial\Omega)
\left( |x-y|^{-1} + d(x,\partial \Omega)^{-1}\right).
\end{equation}
In order to exploit the symmetry of $G_x (y)$ we need a similar estimate also
for $| \nabla_x G_x (y) |$. To this end one has to differentiate (\ref{1.3})
with respect to the parameter (!) $x$ and obtains as before that
for  $d(y,\partial \Omega)\le d(x,\partial \Omega)$
\begin{equation}\label{1.64}
| \nabla_x G_x (y) | \le  C(C_0,n,R,\lambda,\|a\|_{C^{0,\alpha}},\partial\Omega)
\left( |x-y|^{-1} + d(x,\partial \Omega)^{-1}\right).
\end{equation}
By symmetry $G_x (y)=G_y(x)$, and (\ref{1.64}) shows that  for
$d(x,\partial \Omega)\le d(y,\partial \Omega)$, one has
\begin{equation}\label{1.65}
| \nabla_y G_x (y) | \le  C(C_0,n,R,\lambda,\|a\|_{C^{0,\alpha}},\partial\Omega)
\left( |x-y|^{-1} + d(y,\partial \Omega)^{-1}\right)
\end{equation}
while (\ref{1.63})  yields
\begin{equation}\label{1.66}
| \nabla_x G_x (y) | \le  C(C_0,n,R,\lambda,\|a\|_{C^{0,\alpha}},\partial\Omega)
\left( |x-y|^{-1} + d(y,\partial \Omega)^{-1}\right).
\end{equation}
Combining (\ref{1.63})-(\ref{1.66}) proves  (\ref{1.01}) and hence (\ref{1.00}) in the case $n=4$.

\bigskip\noindent
{\it Case $n=3$.} Here, we simply work with the bounded Lipschitz
continuous fundamental solution
\begin{equation}\label{1.68}
\Gamma_0 (x,y):=-\frac{1}{8\pi }|x-y|
\end{equation}
so that no iterative procedure is needed and we may directly work with $\ell=0$.
One comes up with
\begin{equation}\label{1.69}
\| u_x\|_{C^{1,\alpha}\left( \overline{\Omega} \right)}
\le C(R,n,\lambda,\|a\|_{C^{0,\alpha}},\partial\Omega)  d(x,\partial \Omega)^{-\alpha}.
\end{equation}
Proceeding as for $n=4$ yields (\ref{1.01}) and hence (\ref{1.00}) also in the case $n=3$.
\qed\end{proof}

Let us now show that assuming certain uniform estimates on the Green's  functions
$H_k$ for the biharmonic operators on a family $(\Omega_k)$ of domains
according to Definition~\ref{definition1} implies similar uniform estimates
for the Green's  functions of the perturbed biharmonic operators $\Delta^2+a_k$
on $\Omega_k$:

\begin{proposition}\label{proposition2}
Let $n\ge 4$ and $(\Omega_k)_{k\in\mathbb{N}}$ be a
$C^{4,\alpha}$-smooth perturbation of the bounded
$C^{4,\alpha}$-smooth  domain $\Omega$ according to
Definition~\ref{definition1} and $R>0$ such that $\Omega_k\subset
B_R(0)$. Let $H_k\in C^4\left(
\overline{\Omega_k}\times\overline{\Omega_k} \setminus \{ (x,y):
x=y\}\right)$  denote the Green's  functions for $\Delta^2$ in
$\Omega_k$ and assume that there exists a uniform constant $C_1$
such that for all $k$ and all $x,y\in \Omega_k$ $(x\not=y)$
\begin{equation}\label{1.30}
 |H_k (x,y)| \le C_1\cdot
\left\{\begin{array}{ll}
 |x-y|^{4-n},\quad&\mbox{\ if\ } n>4,\\
\left(1+\left|\log |x-y|\right|\right),\quad&\mbox{\ if\ } n=4.
\end{array}\right.
\end{equation}
Let $a_k\in C^{0,\alpha} \left( \overline{\Omega_k}\right)$ and
$\Lambda>0$ such that $\forall k:\| a_k\|_{C^{0,\alpha}
 (\overline{\Omega_k})}\le \Lambda$
and let $\lambda>0$ be such that
$$
\forall k\quad  \forall \varphi\in C^{\infty}_c (\Omega_k) :\qquad
\int_{\Omega_k} \left( (\Delta\varphi)^2+a_k\varphi^2\right)\, dy \ge
         \lambda\int_{\Omega_k}\varphi^2\, dy.
$$
We denote by $G_k$ the Green's  functions for $\Delta^2+a_k$ in $\Omega_k$.
Then, there exists a constant $C_2=C_2(R,n,C_1,\lambda,\Lambda,\Omega)$ such that
one has the following estimate:
\begin{equation}\label{1.35}
\forall x,y\in \Omega_k, x\not=y:\quad |G_k (x,y)|\le C_2
\cdot
\left\{\begin{array}{ll}
 |x-y|^{4-n},&\mbox{\ if\ } n>4,\\
\left(1+\left|\log |x-y|\right|\right),&\mbox{\ if\ } n=4.
\end{array}\right.
\end{equation}
Moreover, assuming
\begin{equation}\label{1.351}
\forall x,y\in \Omega_k, x\not=y:\quad |\nabla_{(x,y)} H_k (x,y)|\le C_1
\begin{array}{ll}
\left|x-y\right|^{-1},\quad&\mbox{\ if\ } n=4,
\end{array}
\end{equation}
in dimension $n=4$ implies that
\begin{equation}\label{1.352}
\forall x,y\in \Omega_k, x\not=y:\quad |\nabla_{(x,y)} G_k (x,y)|\le C_2
\begin{array}{ll}
\left|x-y\right|^{-1},\quad&\mbox{\ if\ } n=4.
\end{array}
\end{equation}
\end{proposition}

The dependence on $(\Omega_k)_k$
as regular perturbations of $\Omega$
is explicit via the geometric properties of
$\partial\Omega$. As long as these properties are uniformly satisfied,
the same constant may be chosen.

The case $n=3$ need not be covered here, since in this case, Proposition~\ref{proposition1}
already provides strong enough information for our purposes.

\begin{proof}
We proceed quite similarly as in the proof of Proposition~\ref{proposition1},
but now using the biharmonic Green's  functions $H_k$ instead of $\Gamma_0$.
That means that in $\Omega_k$, we define inductively
$$
\Gamma_{k,1} (x,y):=-\int_{\Omega_k} H_{k} (x,z)a_k(z)H_k(z,y)\, dz;
$$
$$
\Gamma_{k,j+1} (x,y):=-\int_{\Omega_k} \Gamma_{k,j} (x,z)a_k(z)H_k(z,y)\, dz.
$$
Moreover, as above, we make the ansatz with $u_{k,x}\in C^{4,\alpha}(\overline{\Omega_k})$
\begin{equation}\label{1.10}
G_k (x,y) :=H_k (x,y)+\sum^{\ell}_{j=1} \Gamma_{k,j}(x,y)+u_{k,x}(y).
\end{equation}
We choose  $\ell>\frac{n}{4}+1$ so that
\begin{equation}\label{1.13}
|\Gamma_{k,\ell}|\le C(R,n,\Lambda),\qquad
|\nabla \Gamma_{k,\ell}|\le C(R,n,\Lambda),
\end{equation}
while for the other $\Gamma_j$, we have in  particular that
\begin{equation}\label{1.14}
|\Gamma_{k,j} (x,y)|
\le C(R,n,\Lambda)\cdot
\left\{\begin{array}{ll}
 |x-y|^{4-n},\quad&\mbox{\ if\ } n>4,\\
\left(1+\left|\log |x-y|\right|\right),\quad&\mbox{\ if\ } n=4,
\end{array}\right.
\end{equation}
and assuming (\ref{1.351}) that
\begin{equation}\label{1.1212}
\forall x,y\in \Omega_k, x\not=y:\quad |\nabla_{(x,y)}\Gamma_{k,j}  (x,y)|\le
C(R,n,\Lambda) \cdot
\left|x-y\right|^{-1},\mbox{\ if\ } n=4.
\end{equation}
As before we see that $G_k$ is
indeed the Green's  function for the Dirichlet problem for $\Delta^2+a_k$
in $\Omega_k$,
iff the $u_{k,x}$ solve the following boundary value problems:
\begin{equation}\label{1.11}
\left\{ \begin{array}{ll}
\Delta^2 u_{k,x}(y)+a_k(y) u_{k,x}(y)=-a_k(y) \Gamma_{k,\ell} (x,y)\quad&\mbox{\ in\ }\Omega_k\\
u_{k,x}(y)=\frac{\partial}{\partial \nu}u_{k,x} =0&\mbox{\ for\ }
                        y\in \partial\Omega_k.
\end{array}\right.
\end{equation}
The right hand side is uniformly bounded, the operators are uniformly
coercive. Hence, $L^p$-theory (see \cite{ADN}) combined with Sobolev embedding theorems
and differentiating (\ref{1.11}) with respect to the parameter $x$
yields
\begin{equation}\label{1.12}
\begin{array}{rcl}
|u_{k,x}(y)|&\le& C(R,n,C_1,\lambda,\Lambda,(\Omega_k)_{k\in\mathbb{N}}),\\
|\nabla_{(x,y)} u_{k,x}(y)| &\le & C(R,n,C_1,\lambda,\Lambda,(\Omega_k)_{k\in\mathbb{N}}).
\end{array}
\end{equation}
The dependence on $(\Omega_k)_k$ is uniform in the sense described before the present proof.
Inserting (\ref{1.13}),  (\ref{1.14}), (\ref{1.12}) and (\ref{1.1212})  into (\ref{1.10}) proves the
claim.
\qed\end{proof}

Finally, we need a more precise statement concerning the smoothness
of the Green's  functions simultaneously with respect to {\em both} variables.

\begin{proposition}\label{proposition2.1}
Under the assumptions of Proposition~\ref{proposition2} we have in addition
that
$$
G_k\in C^{4,\alpha} \left(\overline{\Omega_k}\times \overline{\Omega_k}
\setminus \{ (x,y):\ x\not=y \}\right).
$$
\end{proposition}

\begin{proof}
We let $i\in\{0,\ldots,3\}$ and $p\in (n,n+1)$ so that  in particular
$4-i-\frac{n}{p}>0$. We let $\varphi\in C_c^\infty(\Omega_k)$ and consider
$\psi\in C^{4,\alpha}(\overline{\Omega_k})$  such that
$\Delta^2\psi+a_k\psi=\varphi$ in $\Omega_k$ and $\psi=\partial_\nu\psi=0$  on
$\partial\Omega_k$.  It follows from regularity theory (see \cite{ADN}) and
Sobolev's embedding theorem that  $$
\Vert \psi\Vert_{C^{i,\beta}(\overline{\Omega_k})}\leq C \Vert \psi\Vert_{W^{4,p}(\Omega_k)}
\leq C \Vert \varphi\Vert_{L^p(\Omega_k)}
$$
with $\beta\le 4-i-\frac{n}{p}$, $\beta\in (0,1)$. Here $W^{4,p}$
denotes the Sobolev space of order 4 in differentiability and of
order $p$ in integrability. Since $\psi(x)=\int_{\Omega_k} G_k(x,y)\varphi(y)\,
dy$, we get that $\nabla_x^i G_k$ makes sense and that
$$
\left|\int_{\Omega_k} (\nabla_x^i G_k(x,y)-\nabla_x^iG_k(x',y))\varphi(y)\, dy\right|
\leq C_2 \Vert \varphi\Vert_{L^p(\Omega_k)}|x-x'|^\beta.
$$
By duality, we then get that $y\mapsto \nabla_x^iG_k(x,y)\in L^q(\Omega_k)$ for all
$q\in (\frac{n+1}{n},\frac{n}{n-1})$ and that
$$
\Vert \nabla_x^iG_k(x,\cdot)-\nabla_x^iG_k(x',\cdot)\Vert_{q}
\leq C(q)|x-x'|^\beta\hbox{ for all }x,x'\in\Omega_k.
$$
It follows from the equation satisfied by $G_k(x,\cdot)$ that we have 
$\Delta^2 \nabla_x^iG_k(x,\cdot)+a \nabla_x^iG_k(x,\cdot)=0$ in
${\mathcal D}'(\Omega_k\setminus\{x\})$ and $\nabla_x^iG_k(x,\cdot)=0$,
$\partial_\nu \nabla_x^iG_k(x,\cdot)=0$ on $\partial\Omega_k$. It then follows from regularity
theory that $\nabla_x^iG_k(x,\cdot)\in C^{4,\alpha}(\overline{\Omega_k}\setminus\{x\})$.
Moreover, for all $\delta>0$, there exists $C(\delta)>0$ such that
$$\Vert \nabla_x^iG_k(x,\cdot)-\nabla_x^iG_k(x',\cdot)\Vert_{C^{4,\alpha}
(\overline{\Omega_k}\setminus (B_\delta(x)\cup B_\delta(x')))}
\leq C(\delta)|x-x'|^\beta
$$
for all $x,x'\in\Omega_k$.
This is valid for $i\leq 3$; using the symmetry of the Green's function,
we have a similar result for $i=4$ with respect to the 
$C^{3,\alpha}
(\overline{\Omega_k}\setminus (B_\delta(x)\cup B_\delta(x')))$-norm. 
It then follows that all derivatives of order $4$ are covered so that
$G_k\in C^{4,\alpha}(\overline{\Omega_k}\times\overline{\Omega_k}
\setminus\{(x,x):\, x\in \overline{\Omega_k}\})$.
This proves the proposition.
\qed\end{proof}


\section{Uniform bounds for the Green's  functions}
\label{section4}

As before, we consider a family of bounded regular domains
$(\Omega_k)_{k\in\mathbb{N}}$ being a smooth perturbation of a
fixed bounded regular  domain $\Omega$ according to
Definition~\ref{definition1}. We focus on proving
$$
|H_k (x,y)| \le C_1 \left\{\begin{array}{ll}
|x-y|^{4-n},\quad & \mbox{\ if\ }n>4,\\
\left( 1+|\log|x-y|\,| \right),& \mbox{\ if\ }n=4,\\
1,& \mbox{\ if\ }n=3;\\
\end{array}\right.
$$
$$
|\nabla_{(x,y)} H_k (x,y)| \le C_1 \left\{\begin{array}{ll}
|x-y|^{-1},& \mbox{\ if\ }n=4,\\
1,& \mbox{\ if\ }n=3;\\
\end{array}\right.
$$
with the constant $C_1=C_1(\Omega)$ being uniform for the whole
family $(\Omega_k)_{k\in\mathbb{N}}$. Originally, this type of estimates on the Green's functions is due to Krasovski\u{\i} \cite{Krasovskij2} (cf.
also \cite{Krasovskij1}) even for very general boundary value
problems for even order elliptic operators. For the reader's 
convenience, we include here an independant and shorter proof of these estimates.

\begin{theorem} \label{theorem3}
Let $\Omega$ be a  bounded  $C^{4,\alpha}$-smooth  domain
of $\rn$, $n\geq 3$ and   $(\Omega_k)_{k\in\nn}$  a
$C^{4,\alpha}$-smooth perturbation of $\Omega$. We denote by $H_k$
the Green's  functions for $\Delta^2$ in $\Omega_k$ under
Dirichlet boundary conditions.

Then, there exists a constant $C_1>0$
 such that for all $k$ and all $x,y\in\Omega_k$ with $x\neq y$ one has that
\begin{equation}\label{4.1}
|H_k(x,y)|\leq C_1 \cdot
\left\{\begin{array}{ll}
 |x-y|^{4-n},\quad&\mbox{\ if\ } n>4,\\
\left(1+\left|\log |x-y|\right|\right),&\mbox{\ if\ } n=4,\\
1,\quad&\mbox{\ if\ } n=3.
\end{array}\right.
\end{equation}
Moreover, for $n=3,4$ we prove that
\begin{equation}\label{4.101}
\forall x,y\in \Omega_k, x\not=y:\quad |\nabla_{(x,y)} H_k (x,y)|\le C_1
\cdot
\left\{\begin{array}{ll}
\left|x-y\right|^{-1},&\mbox{\ if\ } n=4,\\
1,&\mbox{\ if\ } n=3.
\end{array}\right.
\end{equation}
\end{theorem}

\begin{proof} If $n=3$, the statement of Proposition~\ref{proposition1}
is already strong enough and nothing remains to be proved. We postpone
the case $n=4$ and start with proving the theorem in the generic
case $n>4$. We argue by contradiction and assume that there exist two
sequences $(x_k)_{k\in\mathbb{N}}, (y_k)_{k\in\mathbb{N}}$
with $x_k,y_k\in\Omega_{\ell_k} $
such that $x_k\neq y_k$ for all $k\in\mathbb{N}$ and such that
\begin{equation}\label{hyp:abs}
\lim_{k\to +\infty}|x_k-y_k|^{n-4}|H_{\ell_k}(x_k,y_k)|=+\infty.
\end{equation}
It is enough to consider $\ell_k=k$; other situations may be reduced to this
by relabelling or are even more special.
After possibly passing to a subsequence,
it follows from \eqref{1.00} that there exists $x_{\infty}\in \partial\Omega$ such that
\begin{equation}\label{lim:xk}
\lim_{k\to +\infty}x_k=x_{\infty}\hbox{ and }
\lim_{k\to +\infty}\frac{d(x_k,\partial\Omega_k)}{|x_k-y_k|}=0.
\end{equation}
We remark that the constant in \eqref{1.00} can be chosen uniformly for the family
 $(\Omega_k)_{k\in\nn}$.

\begin{lemma}\label{lemma4.1} Assume that $n\ge 4$.
For any $q\in \left(\frac{n}{n-3},\frac{n}{n-4}\right)$, there exists $C(q)>0$ such that
for all $k$ and all $x\in\Omega_k$ we have
\begin{equation}\label{ineq:1}
\Vert H_k(x,\, . \, )\Vert_{L^q(\Omega_k)}\leq C d(x,\partial\Omega_k)^{4-n+\frac{n}{q}}.
\end{equation}
The constant $C$ can be chosen uniformly for the family $(\Omega_k)_{k\in\nn}$.
\end{lemma}

\begin{proof}
We proceed with the help of a duality argument.
Let $\psi\in C^\infty_c(\Omega_k)$ and let $\varphi\in C^{4,\alpha}(\overline{\Omega_k})$
be a solution of
$$
\left\{\begin{array}{ll}
\Delta^2\varphi=\psi \quad &\mbox{\ in\ }\Omega_k,\\
\varphi=\partial_\nu\varphi=0&\mbox{\ on\ }\partial\Omega_k.
\end{array}\right.
$$
Let $q\in \left(\frac{n}{n-3},\frac{n}{n-4}\right)$ and denote $q'=\frac{q}{q-1}$
the dual exponent, so that in particular $\frac{n}{4}<q'<\frac{n}{3}$.
It follows from elliptic estimates \cite[Thm. 15.2]{ADN} that there exists
$C_3>0$ independent of $\varphi,\psi$ and $k$ such that
$$
\Vert\varphi\Vert_{W^{4,q'}(\Omega_k)}\leq
C_3\Vert\psi\Vert_{L^{q'}(\Omega_k)}.
$$
The  embedding
$W^{4,q'}(\Omega_k)\hookrightarrow C^{0,\beta}(\overline{\Omega_k})$
with $\beta=4-\frac{n}{q'}=4-n+\frac{n}{q}$
being continuous uniformly in $k$
shows that there exists
$C_4>0$ independent of $\varphi$ and $k$ such that
$\Vert\varphi\Vert_{C^{0,\beta}(\overline{\Omega_k})}\leq
C_4\Vert\varphi\Vert_{W^{4,q'}(\Omega_k)}$. Let $x\in\Omega_k$ and
$x'\in\partial\Omega_k$. We then get that
$$
|\varphi(x)|=|\varphi(x)-\varphi(x')|\leq
\Vert\varphi\Vert_{C^{0,\beta}(\overline{\Omega_k})}
|x-x'|^\beta\leq
C_3C_4\Vert\psi\Vert_{L^{q'}(\Omega_k)}|x-x'|^\beta.$$
Moreover, $\varphi(x)=\int_{\Omega_k} H_k (x,y)\psi(y)\, dy$ for all $x\in\Omega_k$
by Green's  representation formula.
Therefore, taking the infimum with respect to
$x'\in\partial\Omega_k$, we have that
$$
\left|\int_{\Omega_k} H_k (x,y)\psi(y)\, dy\right|
\leq C_3 C_4\Vert\psi\Vert_{L^{q'}(\Omega_k)}d(x,\partial\Omega_k)^\beta
$$
for all $\psi\in C^\infty_c(\Omega_k)$. Inequality \eqref{ineq:1} then
follows.
\qed\end{proof}

\begin{lemma}\label{lemma4.2}
Assuming $n>4$ and (\ref{hyp:abs}), one has that
$\lim_{k\to +\infty}|x_k-y_k|=0$.
\end{lemma}

\begin{proof}
Assume by contradiction that $|x_k-y_k|$ does not converge to $0$. After
extracting a subsequence we may then assume that there exists $\delta >0$
such that all $x_k\in B_{\delta} (x_{\infty})$ and all $y_k\in \Omega_k
\setminus \overline{B_{3\delta} (x_{\infty})}$. We consider $q$ as in Lemma~\ref{lemma4.1}.
In particular we know that $\Vert H_k(x,\, . \, )\Vert_{L^q(\Omega_k)}\leq C$
uniformly in $k$. By applying
local elliptic estimates (cf. \cite[Theorem 15.3]{ADN})
combined with Sobolev embeddings
in $\Omega_k \setminus \overline{B_{2\delta} (x_{\infty})}$
we find that
$$
\| H_k(x_k,\, . \, )\|_{L^{\infty}\left(\Omega_k\setminus \overline{B_{3\delta} (x_{\infty})}\right)}
\le C(q,\delta)
$$
uniformly in $k$. In particular, we would have
$$
| H_k(x_k,y_k ) |\le C(q,\delta)\quad \mbox{\ and\ }\quad
|x_k-y_k|^{n-4}| H_k(x_k,y_k ) |\le C(q,\delta)
$$
independent of $k$. This contradicts our hypothesis (\ref{hyp:abs}).
\qed\end{proof}

\medskip\noindent{\it Concluding the proof of Theorem~\ref{theorem3},
case $n>4$.}
In what follows we may work in one fixed coordinate
domain $U_i$; for this reason we drop the index $i$. Let
$\Phi_k:U\to \mathbb{R}^n$ be coordinate charts of $\Omega_k$ at
$x_{\infty}$ as in Definition~\ref{definition1}.  We recall that
 $$
\Phi_k(U\cap \{x_1<0\})=\Phi_k(U)\cap\Omega_k\hbox{ and }
\Phi_k(U\cap
\{x_1=0\})=\Phi_k(U)\cap\partial\Omega_k.
$$
Without loss of generality we may assume that
$\Phi_k(0)=x_{\infty}$ and $B_\delta(x_{\infty})\subset\Phi_k(U)$.

We let $x_k=\Phi_k({x}_k')$ and $y_k=\Phi_k({y}_k')$.
Therefore, \eqref{lim:xk} rewrites as
\begin{equation}\label{lim:nxk}
\lim_{k\to +\infty} {x}_k'=0\hbox{ and }
\lim_{k\to +\infty}\frac{{x}_{k,1}' }{|{x}_k'-{y}_k'|}=0.
\end{equation}
We define for $R$ large enough
$$
\tilde{H}_k(z)=|{x}_k'-{y}_k'|^{n-4}H_k(\Phi_k({x}_k'),
 \Phi_k({x}_k'+|{x}_k'-{y}_k'|(z-\rho_k \vec{e}_1)))
$$
in $B_R(0)\cap\{x_1<0\}$, where $\rho_k:=\frac{ {x}_{k,1}'
}{|{x}_k'-{y}_k'|}$.  We rewrite the biharmonic equation $\Delta^2 H_k (x,\, .
\, )=0$  complemented with Dirichlet boundary conditions as
$$
\Delta_{g_k}^2 \tilde{H}_k=0
\hbox{ in }\left( B_R(0)\cap\{z_1<0\}\right) \setminus\{\rho_k\vec{e}_1\},
\quad \tilde{H}_k=\partial_1\tilde{H}_k=0\hbox{ on }\{z_1=0\}.
$$
Here, $g_k (z)=\Phi_k^*({\mathcal E})
({x}_k'+|{x}_k'-{y}_k'|(z-\rho_k \vec{e}_1))$, ${\mathcal E}=(\delta_{ij})$
the Euclidean metric, and
$\Delta_{g_k}$ denotes the Laplace-Beltrami operator with respect
to this rescaled and translated
pull back of the Euclidean metric under $\Phi_k$.
Then, for $\tau>0$ being chosen suitably small,
it follows from elliptic estimates (see \cite[Theorem 15.3]{ADN})
and Sobolev embeddings
that there exists $C(R,q,\tau)>0$ such that
\begin{equation}\label{est:Gk}
|\tilde{H}_k(z)|\leq C(R,q,\tau)\Vert \tilde{H}_k\Vert_{L^q(B_{R}(0)\setminus
B_{\tau}(0))} \end{equation}
for all $z\in B_{R/2}(0)\setminus B_{2\tau}(0)$, $z_1\le 0$.
In order to estimate the $L^q$-norm on the right-hand side we use
  \eqref{ineq:1} and obtain  that
\begin{eqnarray*}
\int_{B_R(0)\cap\{\zeta_1<0\}}|\tilde{H}_k(\zeta)| ^q\, d\zeta
&\leq &C|{x}_k'-{y}_k'|^{q(n-4)-n}\int_{\Omega_k}|H_k(x_k,y)|^q\, dy\\
&\leq &C|{x}_k'-{y}_k'|^{q(n-4)-n} d(x_k,\partial\Omega_k)^{(4-n)q+n}\\
&\leq
&C\left(\frac{d(x_k,\partial\Omega_k)}{|{x}_k'-{y}_k'|}\right)^{n-q(n-4)}.
\end{eqnarray*}
Therefore, with \eqref{lim:xk}, we get that
$\lim_{k\to +\infty}\Vert \tilde{H}_k\Vert_{L^q(B_{R}(0)\setminus
B_{\tau}(0))}=0$,  and \eqref{est:Gk} yields
$$
\lim_{k\to +\infty}\tilde{H}_k=0\hbox{\ in\ }C^0((B_{R/2}(0)\setminus
B_{2\tau}(0)) \cap\{ z_1 \le 0 \}).
$$
In particular, since $\lim_{k\to +\infty}\rho_k=0$, we have that
\begin{equation*}
\lim_{k\to
+\infty}\tilde{H}_k\left(\frac{{y}_k'-{x}_k'}{|{y}_k'-{x}_k'|}
+\rho_k \vec{e}_1\right)=0. \end{equation*}
This limit rewrites as
$$\lim_{k\to
+\infty}|x_k-y_k|^{n-4}|H_k(x_k,y_k)|=0,$$
contradicting \eqref{hyp:abs}. The proof of Theorem~\ref{theorem3},
$n>4$, is complete.
\qed\end{proof}

\medskip\noindent{\it Proof of Theorem~\ref{theorem3},
case $n=4$.} Here it is enough to prove (\ref{4.101}) for $\nabla_y$,
exploiting the symmetry of the Green's function. We argue by contradiction and
as in the proof for $n>4$, we may assume that  there exist two
sequences $(x_k)_{k\in\mathbb{N}}, (y_k)_{k\in\mathbb{N}}$
with $x_k,y_k\in\Omega_{k} $
such that $x_k\neq y_k$ and
\begin{equation}\label{4.2}
\lim_{k\to +\infty}|x_k-y_k|\, |\nabla_y H_{k}(x_k,y_k)|=+\infty.
\end{equation}
After possibly passing to a subsequence,
it follows from \eqref{1.01} that there exists $x_{\infty}\in \partial\Omega$ such that
\begin{equation}\label{4.3}
\lim_{k\to +\infty}x_k=x_{\infty}\hbox{ and }
\lim_{k\to +\infty}\frac{d(x_k,\partial\Omega_k)}{|x_k-y_k|}=0.
\end{equation}

Lemma~\ref{lemma4.1} may be applied with some $q>4$.
The analogue of Lemma~\ref{lemma4.2} is proved in exactly the same way as above.
Like above we now put for $R$ large enough
$$
\tilde{H}_k(z)=H_k(\Phi_k({x}_k'),
 \Phi_k({x}_k'+|{x}_k'-{y}_k'|(z-\rho_k \vec{e}_1)))
$$
in $B_R(0)\cap\{z_1<0\}$, where
$x_k=\Phi_k({x}_k')$, $y_k=\Phi_k({y}_k')$,
$\rho_k:=\frac{ {x}_{k,1}' }{|{x}_k'-{y}_k'|}$.
As above we find for $\tau>0$ small enough
that there exists $C(R,\tau,q)>0$ such that
\begin{equation*}
|\nabla \tilde{H}_k(z)|\leq C(R,q,\tau)\Vert
\tilde{H}_k\Vert_{L^q(B_{R}(0)\setminus B_{\tau}(0))}
 \end{equation*}
for all $z\in B_{R/2}(0)\setminus B_{2\tau}(0)$, $z_1\le 0$.
Using (\ref{ineq:1}) we obtain  that
\begin{eqnarray*}
\int_{B_R(0)\cap\{\zeta_1<0\}}| \tilde{H}_k(\zeta)| ^q\, d\zeta
&\leq &C|{x}_k'-{y}_k'|^{-4}\int_{\Omega_k}|H_k(x_k,y)|^q\, dy\\
&\leq &C\left(\frac{d(x_k,\partial\Omega_k)}{|{x}_k'-{y}_k'|}\right)^{4}.
\end{eqnarray*}
In the same way as in the generic case $n>4$, this yields first that
$$
\lim_{k\to +\infty}\nabla \tilde{H}_k=0\hbox{\ in\ }C^0((B_{R/2}(0)\setminus
B_{2\tau}(0)) \cap\{ z_1 \le 0 \})
$$
and back in the original coordinates
$$
\lim_{k\to\infty} |x_k-y_k| \, \left| \nabla_y H_k (x_k,y_k) \right|=0.
$$
So, we achieve a contradiction also if $n=4$. This proves
\eqref{4.101}. Integrating \eqref{4.101}, we get
\eqref{4.1}. The
proof of Theorem~\ref{theorem3} is complete. \hfill $\square$


\section{Convergence of the Green's  functions}
\label{section2}

As before, we consider a family of bounded regular domains
$(\Omega_k)$ being a $C^{4,\alpha}$-smooth perturbation of a fixed
bounded  $C^{4,\alpha}$-smooth  domain $\Omega$ according to
Definition~\ref{definition1}. We consider the operators
$\Delta^2+a_k$ in $\Omega_k$ and assume that
$$
\exists U_0\supset \overline{\Omega_k}: a_k\in C^{0,\alpha} (U_0);
$$
$$
\exists a_{\infty}\in C^{0,\alpha} (U_0): \lim_{k\to \infty} a_{k} = a_{\infty}
\mbox{\ in\ } C^{0,\alpha}_{loc} (U_0).
$$
As before, we denote by $G_k$ the Green's  functions corresponding
to $\Delta^2+a_k$ in $\Omega_k$ and by $G$ the Green's  functions
corresponding to $\Delta^2+a_\infty$ in $\Omega$ and show the
following convergence result. 
As for the diffeomorphisms $\Phi_{k,i},\Phi_i$ we refer to
Definition~\ref{definition1}.

\begin{proposition}\label{proposition3}
Let $x_k\in \Omega_k$ and assume that $\lim_{k\to\infty} x_k=x_{\infty}
\in \Omega$.
Then, we have:
\begin{eqnarray*}
G_k (x_k,\, .\, ) &\to & G(x_{\infty},\, .\, )  \mbox{\ in\ } C^4_{loc} (\Omega\setminus\{x_{\infty}\}),\\
G_k (x_k,\, .\, ) &\to & G(x_{\infty},\, .\, )  \mbox{\ in\ }L^1(\mathbb{R}^n),\\
G_k (x_k,\, .\, ) \circ \Phi_{k,i}&\to &
G(x_{\infty},\, .\, ) \circ \Phi_i \mbox{\ in\ }C^4_{loc} (U_i\cap \{ z_1\le 0 \}
           \setminus \{ \Phi_i^{-1} (x_{\infty}) \}).
\end{eqnarray*}
If $n=3$ we have in addition that
$$
G_k (\, .\, ,\, .\, ) \to  G(\, .\, ,\, .\, )  \mbox{\ in\ } C^0_{loc} (\Omega\times \Omega ).
$$
\end{proposition}

\begin{proof} According to Theorem~\ref{theorem3}
and Proposition~\ref{proposition2} we know that
\begin{eqnarray}
|H_k (x,y)|& \le& C  \cdot
\left\{\begin{array}{ll}
 |x-y|^{4-n},&\mbox{\ if\ } n>4,\\
\left(1+\left|\log |x-y|\right|\right),&\mbox{\ if\ } n=4,\\
1,&\mbox{\ if\ } n=3;
\end{array}\right.
\nonumber \\
 |G_k (x,y)| &\le & C \cdot
\left\{\begin{array}{ll}
 |x-y|^{4-n},&\mbox{\ if\ } n>4,\\
\left(1+\left|\log |x-y|\right|\right),&\mbox{\ if\ } n=4,\\
1,&\mbox{\ if\ } n=3;
\end{array}\right.\label{2.0}
\end{eqnarray}
uniformly in $k$. This shows that in particular
\begin{equation}\label{2.1}
\| G_k (x,\, .\, ) \|_{L^1(\Omega_k)}\le C \mbox{\ uniformly in \ }  k.
\end{equation}
Moreover, since $x_k\to x_{\infty}$, we may assume that all $x_k$
are in a small neighbourhood around $x_\infty$. Refering to the
construction in the proof of Proposition~\ref{proposition1} we see
that the $u_{k,x_k}$ arising there are uniformly bounded in
$C^{4,\alpha} \left( \overline{\Omega_k}\right)$. After selecting
a suitable subsequence we see that for each
$\Omega_0\subset\subset\Omega$ one has $G_k (x_k,\, .\, )\to
\varphi$ in $C^4_{loc}\left(\overline{\Omega_0}\setminus \{x_{\infty}
\}\right)$  and
$G_k (x_k,\, .\, ) \circ \Phi_{k,i}\to \varphi \circ \Phi_i
\mbox{\ in\ }C^4_{loc} (U_i\cap \{ z_1\le 0 \}
           \setminus \{ \Phi_i^{-1} (x_{\infty}) \})$ with
a suitable $\varphi\in C^{4,\alpha} (\overline{\Omega}\setminus \{x_{\infty} \})$.
Thanks to this compactness and the fact that in any case the limit
is the uniquely determined Green's  function, we have convergence
on the whole sequence towards $G(x_{\infty},\, .\, )$.

Finally, since we have pointwise convergence, (\ref{2.0}) allows
for applying Vitali's convergence theorem to show that
$$
G_k (x_k,\, .\, ) \to  G(x_{\infty},\, .\, )  \mbox{\ in\ }L^1(\mathbb{R}^n).
$$
The statement concerning $C^0_{loc} (\Omega\times \Omega )$-convergence
in $n=3$ follows from $|\nabla G_k (\, .\, ,\, .\, )| \le C$, cf. (\ref{1.01}).
\qed\end{proof}

In order to prove Lemma~\ref{lemma3.4} below we also
need a convergence result simultaneous in both variables.

\begin{proposition}\label{proposition3.1}
We have that
$$
G_k (\, .\,  ,\, .\, ) \circ \left( \Phi_{k,i} \times \Phi_{k,j}\right) \to
G (\, .\,  ,\, .\, ) \circ \left( \Phi_{i} \times \Phi_{j}\right)
$$
$
\mbox{\ in\ } C^4_{loc} \left(  (U_i\cap \{ x_1\le 0 \})\times(U_j\cap \{
x_1\le 0 \}) \setminus D_{ij}\right),
$
where
$$
D_{ij}=\{ (x,y)\in U_i\times U_j :\ \Phi_i (x)=\Phi_j (y)\}.
$$
\end{proposition}

\begin{proof}
We
combine the ideas of the proofs of Propositions~\ref{proposition3} and
\ref{proposition2.1}. One should observe that Theorem~\ref{theorem3} and
Proposition~\ref{proposition2} guarantee uniform
$L^1$-bounds  for $H_k$ and $G_k$  as in the proof of Proposition~\ref{proposition3}.
\qed\end{proof}


\section{The limit of the zeros of the Green's  functions}
\label{section3}

We keep the notations of the previous sections.
In order to prove Theorem~\ref{theorem2}, we assume that for every $k$,
there exist
\begin{equation}\label{3.choice}
x_k,y_k\in \Omega_k,\ x_k\not=y_k:\qquad G_k (x_k,y_k) =0.
\end{equation}
After passing to subsequences there exist $x_\infty=\lim_{k\to\infty} x_k,y_\infty=\lim_{k\to\infty} y_k$.
Using Definition~\ref{definition1}, one sees that $x_\infty,y_\infty\in \overline{\Omega}$.

As for the location of these limit points, we distinguish several cases.

\subsection{Both points in the interior}\label{interiorinterior}

Here, we consider the case that  $x_\infty,y_\infty\in \Omega$.
Once it is shown that $x_\infty\not= y_\infty$ we conclude directly
from Proposition~\ref{proposition3} that
\begin{equation}
G(x_\infty,y_\infty)=0.
\end{equation}
So, we are left with proving:

\begin{lemma}\label{lemma3.1}
$x_\infty\not= y_\infty.$
\end{lemma}

\begin{proof} Assume by contradiction that $x_\infty=y_\infty$. We consider
first the case $n>4$ and here,
the rescaled Green's  function:
\begin{equation}\label{3.0}
\tilde{G}_k (z):=|x_k-y_k|^{n-4} G_k (x_k,x_k+|x_k-y_k|z).
\end{equation}
Let $\varepsilon>0$ be such that $\overline{B_{2\varepsilon} (x_\infty)}\subset
\Omega\cap\Omega_k$ for all $k$. Then,
for $k$ large enough, $|x_k-x_{\infty}| <\varepsilon$ and
$\tilde{G}_k (z)$ is certainly defined for
$|z|<\frac{\varepsilon}{|x_k-y_k|}$, where
one has  by Theorem~\ref{theorem3} and Proposition~\ref{proposition2} that
\begin{equation}\label{3.1}
|\tilde{G}_k (z)|\le C |z|^{4-n}
\end{equation}
uniformly in $k$. Because the $\tilde{G}_k $ are defined on a sequence of sets
which exhausts the whole $\mathbb{R}^n$ we may discuss how to pass to a limit
locally in $\mathbb{R}^n$. Since
$$
\Delta^2 \tilde{G}_k +|x_k-y_k|^4a_k(x_k+|x_k-y_k|z)  \tilde{G}_k=0
\mbox{\ on\ } \overline{B_{\varepsilon/|x_k-y_k|}(0)}\setminus \{ 0\},
$$
by elliptic Schauder theory we may assume that after possibly passing to
a subsequence that
\begin{equation}
\tilde{G}_k \to \tilde{G} \mbox{\ in\ } C^4_{loc} (\mathbb{R}^n \setminus\{ 0\}),
\mbox{\ where\ } |\tilde{G}(z)| \le C|z|^{4-n} .
\end{equation}
Moreover,
$$
\Delta^2 \tilde{G}=0\mbox{\ in\ } \mathbb{R}^n \setminus\{ 0\}.
$$
In order to compute the differential equation  satisfied by
$\tilde{G}$ near $z=0$, let $\varphi\in C_c^{\infty} (\mathbb{R}^n)$ with
supp $\varphi\subset B_R (0)$ and define for $k$ large enough
$$
\Omega_k\ni x\mapsto \varphi_k(x):=\varphi\left(\frac{x-x_k}{|x_k-y_k|}\right),
\quad \varphi_k \in C_c^{\infty} (\Omega_k).
$$
\begin{eqnarray*}
\varphi(0) &=& \varphi_k(x_k)=\int_{\Omega_k} G_k(x_k,y)(\Delta^2 \varphi_k
           +a_k \varphi_k) \, dy\\
&=&\int_{B_{|x_k-y_k|R }(x_k)} G_k(x_k,y)|x_k-y_k|^{-4} \\
&&\quad \cdot \left(
      \left(\Delta^2\varphi\right) \left(\frac{y-x_k}{|x_k-y_k|}\right)
         +|x_k-y_k|^4 a_k (y) \varphi\left(\frac{y-x_k}{|x_k-y_k|}\right)
            \right) \, dy\\
&=&\int _{B_R(0)} |x_k-y_k|^{n-4} G_k(x_k,x_k+|x_k-y_k|z)\\
&&\quad \cdot
\left( \Delta^2 \varphi (z) + |x_k-y_k|^4 a_k(x_k+|x_k-y_k|z)\varphi(z)\right)\, dz\\
&=& \int_{\mathbb{R}^n} \tilde{G}_k(z)
\left( \Delta^2 \varphi (z) + |x_k-y_k|^4 a_k(x_k+|x_k-y_k|z)\varphi(z)\right)\, dz.
\end{eqnarray*}
We put $\gamma_n=\frac{1}{2(n-4)(n-2)ne_n}$ and obtain, letting $k\to \infty$:
$$
\int_{\mathbb{R}^n}\tilde{G} (z) \Delta^2 \varphi (z) \, dz =\varphi (0) =
\int_{\mathbb{R}^n}\gamma_n|z|^{4-n} \Delta^2 \varphi (z) \, dz .
$$
This shows that we have in the sense of distributions that
$$
\Delta^2 \left( \tilde{G}(z)-\gamma_n|z|^{4-n}\right) =0
\mbox{\ in\ } \mathbb{R}^n.
$$
Hence,
$$
\tilde{G}(z)=\gamma_n|z|^{4-n} +\psi (z),\quad \psi\in C^{\infty}(\mathbb{R}^n),
\quad \Delta^2 \psi =0.
$$
Thanks to (\ref{3.1}) we know further that
$$
|\psi(z)|\le C(1+|z|)^{4-n}.
$$
Also for entire bounded biharmonic (even more generally polyharmonic) functions,
Liouville's theorem holds true, i.e. these are constant, see \cite[p. 19]{Nicolesco}.
Hence $\psi(z)\equiv 0$ showing that
$$
\tilde{G}(z)=\gamma_n|z|^{4-n},\qquad z\in \mathbb{R}^n\setminus \{ 0\}.
$$
On the other hand we have according to the choice (\ref{3.choice})
of $x_k,y_k$ and the definition (\ref{3.0}) of $\tilde{G}_k$ that
$$
\tilde{G}_k \left( \frac{y_k-x_k}{|x_k-y_k|}\right)=
|x_k-y_k|^{n-4} G_k \left(x_k,y_k\right)=0.
$$
Hence there exists at least one point $\zeta\in\mathbb{R}^n$ with
$$
|\zeta|=1 \mbox{\ and\ } 0=\tilde{G}(\zeta)=\gamma_n|\zeta|^{4-n},
$$
which is false.
This proves the statement for the case $n>4$.
One should observe that when looking just at the biharmonic operator,
a proof for the previous lemma would directly follow from the local
positivity results in general domains, which are proved in \cite{GrunauSweers3}.
This observation will be useful in what follows.

Let us now consider the case $n=4$. Since $x_\infty\in\Omega$, according to
\cite{GrunauSweers3}, there exists (a small) $\delta_1>0$ such that for
all $k$ and all $x,y\in \Omega_k$ we have that
\begin{equation}\label{3.101}
x,y\in B_{\delta_1}(x_\infty) \quad
\Rightarrow \quad H_k(x,y)\ge - \frac{1}{c_3}
\log|x-y|.
\end{equation}
We estimate the difference between $G_k$ and $H_k$. For arbitrary
but fixed $x\in \Omega$, we have that with respect
 to the $y$-variable, $\left(H_k-G_k\right)(x,\,
.\, )$ solves the following Dirichlet problem:
\begin{equation*}
\left\{\begin{array}{ll}
\Delta_y^2 \left(H_k-G_k\right)(x,y)+a_k(y) \left(H_k-G_k\right)(x,y)=
a_k(y) H_k(x,y),\quad&y\in \Omega_k\\
\left(H_k-G_k\right)(x,y)=\frac{\partial}{\partial
\nu_y}\left(H_k-G_k\right)(x,y)=0, &y\in
\partial\Omega_k.
\end{array}\right.
\end{equation*}
According to
Theorem~\ref{theorem3}, we have that $\| a_k(\, .\, ) H_k(x,\, .\,
)\|_{L^2(\Omega_k)}\le c_4$ uniformly in $k$ and $x$. Since
$\Delta^2+a_k$ is assumed to be uniformly coercive, elliptic
estimates from \cite{ADN} show that
\begin{equation*}
\| \left(H_k-G_k\right)(x,\, .\, )\|_{L^{\infty} (\Omega_k)}\le C
\| \left(H_k-G_k\right)(x,\, .\, )\|_{W^{4,2} (\Omega_k)} \le c_5,
\end{equation*}
uniformly in $x$ and $k$. Together with (\ref{3.101}), this gives
that there exist a $\delta_2>0$ and a constant
$c_6>0$  such that
\begin{equation}\label{3.102}
x,y\in B_{\delta_2} (x_\infty)\quad
\Rightarrow \quad G_k(x,y)\ge - \frac{1}{c_6}
\log|x-y|.
\end{equation}
This proves  the claim also for $n=4$, since by (\ref{3.102}), it
is impossible that $G_k(x_k,y_k)=0$, where $x_k,y_k\to
x_\infty\in\Omega$.

Finally, we consider $n=3$. Since here, according to
Proposition~\ref{proposition3}, also $G_k\to G$ in
$C^0_{loc}(\Omega\times\Omega)$ we have by assumption that
$G(x_{\infty},x_{\infty})=0$. On the other hand, testing the
boundary value problem for $G(x_{\infty},\, .\, )$ with
$G(x_{\infty},\, .\, )$ itself yields by virtue of the uniform
coercivity  that
$$
G(x_{\infty},x_{\infty})\ge \lambda  \int_\Omega G(x_{\infty},y)^2\, dy>0.
$$
We obtain a contradiction also in the case $n=3$.
So, the proof of Lemma~\ref{lemma3.1} is complete.
\qed\end{proof}

\subsection{One point in the interior, one point on the boundary}
\label{interiorboundary}

After possibly interchanging $x_\infty$ and $y_\infty$ we may consider the case
that $x_\infty\in \Omega$, $y_\infty\in \partial \Omega$.

\begin{lemma}\label{lemma3.2}
$$
\Delta_y G(x_\infty,y_\infty)=0.
$$
\end{lemma}

\begin{proof}
We may fix a neighbourhood $B_{\delta} (p_i)$ such that $y_\infty\in \partial \Omega\cap
B_{\delta} (p_i)$ so that for $k$ large enough $y_k \in \Omega_k\cap B_{\delta} (p_i)$.
We denote $y_k':=\Phi_{k,i}^{-1} (y_k)$, $y_\infty':=\Phi_{i}^{-1} (y_\infty)$ and
observe that $(y_k')_1<0$, $(y_\infty')_1=0$, $y_k'\to y_\infty'$ in $U_i$. 
we recall the notation $y_k'=( y_{k,1}',\bar{y_k'})$.
Writing
$$
\tilde{G}_{k,i}:=G_k (x_k,\, .\, ) \circ \Phi_{k,i},\quad
\tilde{G}_i:=G_k (x_{\infty},\, .\, ) \circ \Phi_{i}
$$
we see by means of Taylor's expansion that with suitable $\theta_k\in (0,1)$:
\begin{eqnarray*}
0&=& G_k(x_k,y_k)=\tilde{G}_{k,i}(y_k')\\
&=&\tilde{G}_{k,i}(0, \bar{y_k'})+\partial_1\tilde{G}_{k,i}(0,\bar{y_k'}) y_{k,1}'
+\frac{1}{2}\partial_{11}\tilde{G}_{k,i}(\theta_k y_{k,1}',\bar{y_k'})(y_{k,1}')^2\\
&=&\frac{1}{2}\partial_{11}\tilde{G}_{k,i}(\theta_k y_{k,1}',\bar{y_k'})(y_{k,1}')^2
\end{eqnarray*}
due to the boundary conditions on $G_k$. According to Proposition~\ref{proposition3}
this yields $\partial_{11}\tilde{G}_i(y_\infty')=0$. Since $G_k (x_k,\, .\, )|_{\partial \Omega}
=\frac{\partial}{\partial\nu}G_k (x_k,\, .\, )|_{\partial \Omega}=0$, we obtain back in
the original coordinates that $\Delta_y G(x_\infty,y_\infty)=0$ as stated.
\qed\end{proof}

\subsection{Both points on the boundary}
\label{boundaryboundary}

So, here we have to consider the case that both $x_\infty\in\partial\Omega$ and
$y_\infty\in\partial\Omega$. The most delicate part will be to prove that both
points have to be distinct:

\begin{lemma}\label{lemma3.3}
$x_\infty\not= y_\infty$.
\end{lemma}

The proof is rather technical and will be postponed to Subsection~\ref{proof3.3}.
Assuming now Lemma~\ref{lemma3.3} being proved it is not too difficult that
in this case an additional zero of the Green's  function can be observed on the
boundary:

\begin{lemma}\label{lemma3.4}
$
\Delta_x \Delta_y G(x_\infty,y_\infty)=0.
$
\end{lemma}

\begin{proof}
According to Proposition~\ref{proposition2.1} we have that $G\in C^{4,\alpha}$
in a neighbourhood of $(x_\infty,y_\infty)$.
This proof is similar to
that of Lemma~\ref{lemma3.2}. We fix neighbourhoods such that
$x_\infty \in B_\delta (p_i), y_\infty\in B_\delta (p_j)$;
without loss of generality we may assume that $B_\delta (p_i)\cap
B_\delta (p_j)=\emptyset.$ Moreover we may
assume that $\forall k:\quad x_k \in B_\delta (p_i), y_k\in B_\delta (p_j)$.
To work in local charts we define
$$
x_k':=\Phi_{k,i}^{-1} (x_k) ,\quad x_\infty':=\Phi_{i}^{-1} (x_\infty),\quad
y_k':=\Phi_{k,j}^{-1} (y_k) ,\quad y_\infty':=\Phi_{j}^{-1} (y_\infty).
$$
Hence we have
$$
x_k'\in U_i\cap\{x_1<0\},\quad x_k'\to x_\infty'\in U_i\cap\{x_1=0\},
$$
$$
y_k'\in U_j\cap\{y_1<0\}, \quad y_k'\to y_\infty'\in U_j\cap\{y_1=0\}.
$$
Defining
$$
\tilde{G}_k: U_i\cap\{x_1\le 0\}\times U_j\cap\{y_1\le 0\}\to \mathbb{R},\,
\tilde{G}_k (x',y'):= G_k \left( \Phi_{k,i}(x') , \Phi_{k,j}(y') \right);
$$
$$
\tilde{G}: U_i\cap\{x_1\le 0\}\times U_j\cap\{y_1\le 0\}\to \mathbb{R},\,
\tilde{G} (x',y'):= G_k \left( \Phi_{i}(x') , \Phi_{j}(y') \right);
$$
we see that by assumption
$$
0=G_k (x_k,y_k)=\tilde{G}_k (x_k', y_k').
$$
Taylor's expansion  with respect to $y'$
and exploiting the boundary conditions for $\tilde{G}_k$
with respect to $y'$ shows that
for each $k$ there exists a suitable $\theta_k\in(0,1)$ such that
$$
\partial_{y_1}^2 \tilde{G}_k (x_{k,1}',\bar{x_k'},\theta_k y_{k,1}',
\bar{y_k'})=0. $$
Now, we use Taylor's expansion for this expression with respect to $x'$
and obtain with suitable $\tau_k\in (0,1)$:
\begin{eqnarray*}
0&=&  \partial_{y_1}^2 \tilde{G}_k (x_{k,1}',\bar{x_k'},\theta_k
y_{k,1}',\bar{ y_k'}) \\ &=&\partial_{y_1}^2\tilde{G}_k (0,\bar{x_k'},
\theta_k y_{k,1}' ,\bar{ y_k'})          +\partial_{x_1}
\partial_{y_1}^2\tilde{G}_k (0,\bar{x_k'},  \theta_k y_{k,1}' ,\bar{
y_k'})x_{k,1}'\\
    &&\qquad
+\frac{1}{2}\partial_{x_1}^2\partial_{y_1}^2\tilde{G}_k (\tau_k
x_{k,1}',\bar{x_k'},
 \theta_k y_{k,1}' ,\bar{y_k'})(x_{k,1}')^2\\
&=&\frac{1}{2}\partial_{x_1}^2\partial_{y_1}^2\tilde{G}_k
(\tau_k x_{k,1}',\bar{x_k'},              \theta_k y_{k,1}' ,
\bar{y_k'})(x_{k,1}')^2
\end{eqnarray*}
so that
$$
\partial_{x_1}^2\partial_{y_1}^2\tilde{G}_k (\tau_k x_{k,1}',\bar{x_k'},
           \theta_k y_{k,1}' ,\bar{ y_k'})=0.
$$
Since by Proposition~\ref{proposition3.1} we have $C^4$
convergence of $\tilde{G}_k$ to $\tilde{G}$ it follows
that
$$
\partial_{x_1}^2\partial_{y_1}^2\tilde{G}(x_\infty',y_\infty') =0.
$$
Taking into account the boundary conditions of $G$ and of $\tilde{G}$, back in the
original variables we see that
$$
\Delta_x \Delta_y G(x_\infty,y_\infty)=0
$$
thereby proving the claim.
\qed\end{proof}

\subsection{Proof of Lemma~\ref{lemma3.3}}
\label{proof3.3}

We assume by contradiction that $\lim_{k\to\infty}x_k=x_\infty=y_\infty=\lim_{k\to\infty}y_k$.
We choose a neighbourhood $B_{\delta} (p_i)\ni x_\infty$ and may assume
that $\forall k: x_k,y_k\in B_{\delta} (p_i)\cap\Omega_k$. As before we introduce
$$
x_k':=\Phi_{k,i}^{-1} (x_k) ,\quad
y_k':=\Phi_{k,i}^{-1} (y_k), \quad x_\infty':=\Phi_{i}^{-1} (x_\infty)
$$
so that we have
$$
x_k', y_k'\in U_i\cap\{x_1<0\},
\quad x_k'\to x_\infty',\ y_k'\to x_\infty'\in U_i\cap\{x_1=0\}.
$$

We distinguish two further cases according to whether the distance between
$x_k$ and $y_k$ converges faster to $0$ than the distance of these points to
the boundary or vice verca.

\medskip\noindent
{\it First case: $|x_k-y_k|=o\left(\max(d(x_k,\partial
\Omega_k),d(y_k,\partial \Omega_k))\right)$.} After possibly
interchanging $x_k$ and $y_k$ and passing to a subsequence we may
assume that
$$
|x_k-y_k|=o(d(x_k,\partial \Omega_k)).
$$
This case is much simpler than the second case below and quite similar to the
case where both points converge in the interior treated in Subsection~\ref{interiorinterior}.
Like there  we treat the case $n>4$ first.
In this case, we consider the rescaled Green's  functions:
\begin{equation*}
\tilde{G}_k (z):=|x_k-y_k|^{n-4} G_k (x_k,x_k+|x_k-y_k|z).
\end{equation*}
These are is certainly defined for
$|z|<\frac{d(x_k,\partial \Omega_k)}{|x_k-y_k|}$, which converges
to $\infty$ as $k\to \infty$. For this reason, we may now directly copy the
reasoning of Subsection~\ref{interiorinterior} and obtain that
$$
\tilde{G}_k\to\tilde{G} \mbox{\ in\ } C^4_{loc} (\mathbb{R}^n \setminus \{0\})
\mbox{\ with\ } \tilde{G}(z)=\gamma_n|z|^{4-n}.
$$
One should observe that also here the property of the Green's  functions to
be uniformly bounded by $C|x-y|^{4-n}$ is used.
According to the choice (\ref{3.choice})
of $x_k,y_k$ and the definition of $\tilde{G}_k$ we have that
$$
\tilde{G}_k \left( \frac{y_k-x_k}{|x_k-y_k|}\right)=
|x_k-y_k|^{n-4} G_k \left(x_k,y_k\right)=0.
$$
Hence there exists at least one point $\zeta\in\mathbb{R}^n$ with
$$
|\zeta|=1 \mbox{\ and\ } 0=\tilde{G}(\zeta)=\gamma_n|\zeta|^{4-n},
$$
which is false.

We now treat the case $n=4$ and proceed similarly as in the proof
of Lemma~\ref{lemma3.1}. Rescaling the result of
\cite{GrunauSweers3} shows the existence of $\delta>0$,
$c_3>0$  such that for $x,y\in \Omega_k$ with $|x-y|\le \delta
d (x,\partial \Omega_k) $, one has (uniformly in $k$) that
\begin{equation}\label{3.20}
H_k (x,y)\ge-\frac{1}{c_3}\log \frac{|x-y|}{ d (x,\partial \Omega_k)}.
\end{equation}
As it was shown in the proof of Lemma~\ref{lemma3.1}, $G_k-H_k$ is bounded uniformly in $k$.
Hence, there exists a constant $c_4$ such that for  $x,y\in \Omega_k$ we have
$$
|x-y|\le \delta \operatorname{d} (x,\partial \Omega_k) \quad\Rightarrow\quad
G_k (x,y)\ge-\frac{1}{c_3}\log \frac{|x-y|}{ d (x,\partial \Omega_k)}
-c_4.
$$
Since $|x_k-y_k|=o(d(x_k,\partial \Omega_k))$ we obtain
$$
0= G_k (x_k,y_k)\to \infty \qquad (k\to \infty).
$$
This is again false and proves the claim for $n=4$.

Finally we discuss the case $n=3$. Rescaling the result of  Nehari \cite{Nehari}
shows the existence of $\delta>0$, $\varepsilon >0$ such that for $x,y\in \Omega_k$ with
$|x-y|\le \delta  d (x,\partial \Omega_k) $, one has
(uniformly in $k$) that
\begin{equation}\label{3.21}
H_k (x,y)\ge \varepsilon d (x,\partial \Omega_k).
\end{equation}
Making use of elliptic theory as in the proof of Lemma~\ref{lemma3.1} and exploiting
the fact that $n=3$ yields that
$\|\left( G_k(\,.\, ,y_k) -H_k (\,.\, ,y_k)
\right)\|_{C^2(\overline{\Omega_k})} \le c_5$ uniformly in $k$. Since
$|x_k-y_k|\le \delta  d (x_k,\partial \Omega_k) $ for $k$ large enough we
conclude that
$$
0=G_k (x_k,y_k)\ge \varepsilon d (x_k,\partial \Omega_k)-c_6 d (x_k,\partial
\Omega_k)^2 ,
$$
which becomes again false for $k\to \infty$.

\medskip\noindent
{\it Second case: $|x_k-y_k|\not=o(\max(d(x_k,\partial \Omega_k),d(y_k,\partial \Omega_k))$.}
After selecting a subsequence we may assume that there is $\tau>0$ such that
$$
|x_k-y_k|\ge \tau d(x_k,\partial \Omega_k)\mbox{\ and\ }
|x_k-y_k|\ge \tau
d(y_k,\partial \Omega_k). $$
We define
$$
\rho_k :=\frac{(x_k')_1}{|x_k'-y_k'|}<0 \mbox{\ and\ } O(1)
$$
and after  selecting a further subsequence we may assume that
$$
\lim_{k\to\infty} \rho_k =:\rho\le 0.
$$
Again, we will introduce a rescaled family of Green's  functions.
For any $R>0$ and $z,\zeta\in B_R\cap \mathbb{R}^n_-$,
\begin{equation}\label{3.25}
\begin{array}{rcl}
\tilde{G}_k (z,\zeta)&:=& |x_k'-y_k'|^{n-4} G_k \big( \Phi_{k,i}
(x_k'  +|x_k'-y_k'|  (z- \rho_k \vec{e}_1)  ), \\
&&\qquad \qquad\qquad\qquad\Phi_{k,i}
(x_k'  +|x_k'-y_k'|  (\zeta- \rho_k \vec{e}_1)  )
\big).
\end{array}
\end{equation}
Moreover,
$\tilde{G}_k(z,\cdot)=\partial_{\zeta_1}\tilde{G}_k(z,\cdot)=0$ on
$B_R(0)\cap\partial\mathbb{R}^n_-$. According to (\ref{4.1}) and
Proposition~\ref{proposition2}, we see that uniformly in $k$, $z$
and $\zeta$
\begin{equation}\label{6.10.a}
\left| \tilde{G}_k (z,\zeta) \right|  \le C |z-\zeta|^{4-n},\quad
\mbox{\ provided that $n>4$.}
\end{equation}
If $n=3,4$ we conclude first that
$$
\left| \nabla \tilde{G}_k (z,\zeta) \right|  \le C\cdot
\left\{ \begin{array}{ll}
|z-\zeta|^{-1},\quad &\mbox{\ if\ } n=4,\\
1,&\mbox{\ if\ } n=3.
\end{array}\right.
$$
Upon integration  we obtain that
\begin{equation}\label{3.251}
\left|  \tilde{G}_k (z,\zeta) \right|  \le C\cdot
\left\{ \begin{array}{ll}
\left( 1+|\log |z-\zeta| |+ \log(1+ |z|)+ \log(1+|\zeta|) \right),\quad &\mbox{\ if\ } n=4,\\
\left( 1+|z|+ |\zeta|\right) ,&\mbox{\ if\ } n=3.
\end{array}\right.
\end{equation}
The points
$x_k$ and $y_k$ were chosen such that $G_k(x_k,y_k)=0$, which
reads in new coordinates
\begin{equation}\label{3.26}
 \tilde{G}_k \left( \rho_k\vec{e}_1,  \frac{y_k'-x_k'}{|x_k'-y_k'|}+ \rho_k\vec{e}_1\right)=0.
\end{equation}
In order to formulate the differential equation satisfied by $\tilde{G}_k$,
we denote by ${\mathcal E}=\left(\delta_{ij}\right)$ the Euclidean metric
and
$$
g_{k,i} (z):=\Phi_{k,i}^* ({\mathcal E})(x_k'  +|x_k'-y_k'|  (z- \rho_k \vec{e}_1)  )
$$
its translated and rescaled pullback with respect to the coordinate charts $\Phi_{k,i}$.
Moreover, we introduce its limit, the constant metric
$$
g_{\infty,i}:=\Phi_{i}^* ({\mathcal E})(x_\infty).
$$
First, we keep $z\in \mathbb{R}^n_-$ fixed and consider
$\tilde{G}_{k}(z,\, .\, )=:\tilde{G}_{k,z}(\, .\, )$
as function in the second variable. For $\zeta\in B_R(0)\cap \mathbb{R}^n_-\setminus \{z\}$
we have that for $k$ large enough, the following boundary value problem is satisfied:
\begin{equation}
\left\{\begin{array}{l}
\Delta^2_{g_{k,i},\zeta} \tilde{G}_k (z,\zeta)\\
\qquad +|x_k'-y_k'|^4
(a_k\circ \Phi_{k,i})
\left( x_k'  +|x_k'-y_k'|  (\zeta- \rho_k \vec{e}_1)\right)\tilde{G}_k (z,\zeta)=0\\
\hspace*{4.44cm} \mbox{\ for\ } \zeta_1 <0,\ \zeta\not=z,\\
 \tilde{G}_k (z,\zeta)=\partial_{\zeta_1} \tilde{G}_k (z,\zeta)=0\quad \mbox{\ for\ }\zeta_1 =0.
\end{array}\right.
\end{equation}
For $k\to \infty$, using \cite{ADN}, we find $\tilde{G}_z=\tilde{G}(z,\, .\, )\in
                  C^4 \left( \overline{\mathbb{R}^n_-}\setminus \{z\}\right)$
such that
\begin{equation}\label{3.27}
\tilde{G}_{k}(z,\, .\, )\to \tilde{G}_z\mbox{\ in\ }
C^4_{loc} \left(\overline{ \mathbb{R}^n_-}\setminus \{z\}\right),
\quad \Delta^2_{g_\infty,\zeta}\tilde{G}(z,\zeta)=0 \ (z\not=\zeta);
\end{equation}
\begin{equation}\label{3.271}
\left|  \tilde{G} (z,\zeta) \right|  \le C\cdot
\left\{ \begin{array}{ll}
 |z-\zeta|^{4-n},\quad &\mbox{\ if\ } n>4,\\
\left(1+|\log
|z-\zeta|\,|+\log(1+|z|)\,+\log(1+|\zeta|)\,\right),\quad &\mbox{\
if\ } n=4,\\ \left( 1+|z|+|\zeta| \right),&\mbox{\ if\ } n=3;
\end{array}\right.
\end{equation}
\begin{equation}\label{3.272}
\left| \nabla \tilde{G} (z,\zeta) \right|  \le C\cdot
\left\{ \begin{array}{ll}
|z-\zeta|^{-1},\quad &\mbox{\ if\ } n=4,\\
1,&\mbox{\ if\ } n=3.
\end{array}\right.
\end{equation}
In order to calculate the differential equation satisfied by $\tilde{G}$ near $\zeta=z$,
we introduce
$$
\varphi\in C^{\infty}_c \left( \overline{\mathbb{R}^n_-}\right),\qquad
\varphi=\partial_1 \varphi=0 \mbox{\ on\ }\partial\mathbb{R}^n_-
$$
and let $\varphi_k\in C^{4,\alpha} \left( \overline{\Omega_k}\right)$
such that
$$
\varphi(z)=\varphi_k\circ\Phi_{k,i} \left(x_k'  +|x_k'-y_k'|  (z- \rho_k \vec{e}_1)  \right)
\mbox{\ for\ }
z\in \tilde{\Omega}_k
$$
$$
\varphi_k=\partial_{\nu} \varphi_k=0 \mbox{\ on\ } \partial\Omega_k;
$$
where we denote
$$
\tilde{\Omega}_k:=\rho_k\vec{e}_1-\frac{x_k'}{|x_k'-y_k'| } +\frac {1}{|x_k'-y_k'| }
\left( U_i\cap \mathbb{R}^n_-\right) .
$$
By means of the representation formula and the corresponding Green's  function
we see that for $z\in   \mathbb{R}^n_-$ and $k$ large enough
\begin{eqnarray*}
\varphi(z) &=& \varphi_k\left(\Phi_{k,i} \left(x_k'  +|x_k'-y_k'|  (z- \rho_k \vec{e}_1)  \right)\right)\\
&=&\int_{\Omega_k} G_k \left( \Phi_{k,i} \left(x_k'  +|x_k'-y_k'|  (z- \rho_k \vec{e}_1)  \right),
y\right)\left( \Delta^2 \varphi_k +a_k \varphi_k \right)\, dy\\
&=&\int_{\Phi_{k,i}(U_i\cap\{ \eta_1<0  \})}
\hspace*{-1cm}
    G_k \left( \Phi_{k,i} \left(x_k'  +|x_k'-y_k'|  (z- \rho_k \vec{e}_1)  \right),
y\right)\left( \Delta^2 \varphi_k +a_k \varphi_k \right)\, dy\\
&=&\int_{U_i\cap\{ \eta_1<0  \}}G_k \left(
 \Phi_{k,i} \left(x_k'  +|x_k'-y_k'|  (z- \rho_k \vec{e}_1)  \right),
\Phi_{k,i}(\eta)\right)\\
&&\quad \cdot \left(\Delta^2_{\Phi_{k,i}^*({\mathcal E})}
(\varphi_k \circ\Phi_{k,i})
+(a_k \circ \Phi_{k,i})
 (\varphi_k \circ\Phi_{k,i})\right) (\eta)
\left|  \operatorname{Jac}\Phi_{k,i} (\eta)\right|\, d\eta\\
&=&\int_{\tilde{\Omega}_k} |x_k'-y_k'| ^{4-n} \tilde{G}_k (z,\zeta) |x_k'-y_k'|^{-4}\\
&&\quad \cdot
 \left( \Delta^2_{g_{k,i}}  +|x_k'-y_k'|^{4}  a_k \left(\Phi_{k,i}
\left( x_k'  +|x_k'-y_k'|  (\zeta- \rho_k \vec{e}_1)\right)\right)
\right) \varphi (\zeta) \\
&& \quad \cdot |x_k'-y_k'|^n\left|  \operatorname{Jac}\Phi_{k,i} \left(
x_k'  +|x_k'-y_k'|  (\zeta- \rho_k \vec{e}_1)
\right)\right|\, d\zeta\\
&=&\int_{\tilde{\Omega}_k}\tilde{G}_k (z,\zeta)\\
&&\quad \cdot
 \left( \Delta^2_{g_{k,i}}  +|x_k'-y_k'|^{4}  a_k \left(\Phi_{k,i}
\left( x_k'  +|x_k'-y_k'|  (\zeta- \rho_k \vec{e}_1)\right)\right)
\right) \varphi (\zeta) \\
&& \quad \cdot \left|  \operatorname{Jac}\Phi_{k,i} \left(
x_k'  +|x_k'-y_k'|  (\zeta- \rho_k \vec{e}_1)
\right)\right|\, d\zeta.
\end{eqnarray*}
Observing  (\ref{6.10.a}), (\ref{3.251})
and passing to the limit
we obtain for $z\in \mathbb{R}^n_-$:
$$
\varphi (z) =\int_{\mathbb{R}^n_-}\tilde{G}  (z,\zeta) \Delta^2_{g_{\infty,i}} \varphi(\zeta)
\left|  \operatorname{Jac}\Phi_i (x_{\infty}') \right|\, d\zeta.
$$
We introduce the linear bijection $L=d\Phi_i (x_{\infty}')$, the half space $P:=L\left(
\mathbb{R}^n_-\right) $ and obtain for $z\in \mathbb{R}^n_-$:
\begin{equation}\label{3.30}
\varphi (z) =\int_{\mathbb{R}^n_-}\tilde{G}  (z,\zeta) \Delta^2_{L^*{\mathcal E}} \varphi(\zeta)
\left|  \operatorname{det}(L) \right| d\zeta
=\int_{P} \tilde{G}  (z,L^{-1} (\eta) )\Delta^2\left( \varphi\circ L^{-1} \right)  d\eta.
\end{equation}
Finally we consider a rotation $\sigma\in O(n)$ such that $\sigma(P)=\mathbb{R}^n_-$
so that
$\sigma\circ L (\mathbb{R}^n_-)=\mathbb{R}^n_-$.
For arbitrary
$$
\psi\in C^{\infty}_c \left( \overline{\mathbb{R}^n_-}\right),\quad
\mbox{ with }
\psi=\partial_1 \psi=0 \mbox{\ on\ }\partial\mathbb{R}^n_-
$$
and $\tilde{x}\in {\mathbb{R}^n_-}$
we may take $\varphi=\psi\circ \sigma \circ L$ and $z=(\sigma\circ L)^{-1}
(\tilde{x})$. We obtain from (\ref{3.30}) since the Laplacian is invariant
under orthogonal transformations that for $\tilde{x}\in {\mathbb{R}^n_-}$:
\begin{eqnarray*}
\psi (\tilde{x}) &=& \left( \psi\circ \sigma \circ L\right)((\sigma\circ
L)^{-1} (\tilde{x})) \\
&=& \int_{P=\sigma^{-1} ({\mathbb{R}^n_-}) } \tilde{G}
\left((\sigma\circ L)^{-1} (\tilde{x}),L^{-1} (\eta) \right)
\Delta^2\left( \psi\circ \sigma \right)(\eta) \, d\eta\\ 
&=&\int
_{\mathbb{R}^n_-} \tilde{G}  \left((\sigma\circ L)^{-1} (\tilde{x}),
(\sigma\circ
L)^{-1} (\eta)\right)                \Delta^2\psi(\eta) \, d\eta.
\end{eqnarray*} This shows that in the sense of distributions
\begin{equation}\label{3.31}
\Delta^2_{\tilde{y}} \bar{G} (\tilde{x},\, .\, )=\delta_{\tilde{x} },
\end{equation}
where we have defined
\begin{equation}\label{3.32}
\bar{G} (\tilde{x},\tilde{y} ):= \tilde{G}
\left((\sigma\circ L)^{-1} (\tilde{x}),(\sigma\circ L)^{-1} (\tilde{y})\right).
\end{equation}
Moreover, for fixed $\tilde{x}\in  {\mathbb{R}^n_-}$ one concludes
with the help of (\ref{3.271}) and (\ref{3.272}) that
\begin{equation}\label{3.34}
| \bar{G} (\tilde{x},\tilde{y} )| \le C \cdot
\left\{ \begin{array}{ll}
 |\tilde{x}-\tilde{y}|^{4-n},\quad &\mbox{\ if\ } n>4,\\
\left(1+|\log |\tilde{x}-\tilde{y}|\,|+\log(1+|\tilde{x}|)\,
+\log(1+|\tilde{y}|)\,\right),\quad &\mbox{\ if\ }
n=4,\\
\left( 1+|\tilde{x}|+|\tilde{y}| \right),&\mbox{\ if\ } n=3;
\end{array}\right.
\end{equation}
\begin{equation}\label{3.341}
| \nabla \bar{G} (\tilde{x},\tilde{y} )|  \le C\cdot
\left\{ \begin{array}{ll}
|\tilde{x}-\tilde{y}|^{-1},\quad &\mbox{\ if\ } n=4,\\
1,&\mbox{\ if\ } n=3.
\end{array}\right.
\end{equation}
We denote by $H$ the biharmonic Green's  function in $\mathbb{R}^n_-$, which thanks
to Boggio~\cite{Boggio} is known explicitly and known to be positive
-- see Lemma~\ref{lemma3.6} below. We prove:

\begin{lemma}\label{lemma3.5}
$\forall x,y\in  \mathbb{R}^n_-, x\not=y:\qquad  \bar{G} (x,y )=H(x,y )$.
\end{lemma}

\begin{proof}
In what follows we keep $x\in  \mathbb{R}^n_-$ fixed. Both
$\bar{G} (x,\, .\,  )$ and $H(x,\, .\,  )$ satisfy the biharmonic
equation with the $\delta$-distribution $\delta_x$ as right hand
side and zero Dirichlet boundary conditions on $\{ y_1 =0\}$. We
let $\psi_x:=\bar{G} (x,\, .\,  )-H(x,\, .\, )$. Hence,
$$
\psi=\psi_x\in C^{\infty} \left( \overline{\mathbb{R}^n_-} \right)
$$
solves
\begin{equation}\label{3.35}
\left\{ \begin{array}{ll}
\Delta^2 \psi =0 &\mbox{\ in\ } \mathbb{R}^n_-,\\
\psi =\partial_1 \psi =0&\mbox{\ on\ }\{ y_1 =0\}.
\end{array}\right.
\end{equation}
Moreover, according to (\ref{3.34}-\ref{3.341}) and (\ref{3.45}) below we have that
\begin{equation}\label{3.36}
\forall y\in \mathbb{R}^n_- :\qquad
|\psi (y) | \le C  \cdot
\left\{ \begin{array}{ll}
 |y|^{4-n},\quad &\mbox{\ if\ } n>4,\\
\left(1+|\log|y|\,|\right),\quad &\mbox{\ if\ } n=4,\\
\left( 1+|y| \right),&\mbox{\ if\ } n=3;
\end{array}\right.
\end{equation}
\begin{equation}\label{3.361}
| \nabla \psi (y)|  \le C\cdot
\left\{ \begin{array}{ll}
|y|^{-1},\quad &\mbox{\ if\ } n=4,\\
1,&\mbox{\ if\ } n=3;
\end{array}\right.
\end{equation}
where $ C=C(x)$.
According to \cite{Duffin,Huber}
$$
\psi^*(y):=\left\{ \begin{array}{ll}
\psi(y)&\mbox{\ if\ } y_1\le 0,\\
-\psi(-y_1,\bar{y})-2y_1\frac{\partial}{\partial y_1} \psi (-y_1,\bar{y})-
         y_1^2\Delta\psi(-y_1,\bar{y}),&\mbox{\ if\ } y_1 >  0,
\end{array} \right.
$$
$\psi^*\in C^4\left( \mathbb{R}^n\right)$  is an entire biharmonic function. We consider now first the
case $n>4$. Below we will prove that  (\ref{3.35}) and
(\ref{3.36}) imply that also
\begin{equation}\label{3.37}
\forall j=1,2:\,
\forall y\in \mathbb{R}^n_- :\quad
|\nabla^j \psi (y) | \le C |y|^{4-n-j},\, \mbox{\ where\ } C=C(x).
\end{equation}
This immediately gives that $|\psi^* (y) | \le C |y|^{4-n}$ and in
particular that $\psi^*$ is a bounded entire biharmonic
function. Again, Liouville's
theorem for biharmonic functions \cite[p. 19]{Nicolesco} yields
that $\psi^* (y) \equiv 0$ so that the claim of the lemma follows,
provided $n>4$.

If $n=3,4$ we shall prove below that  for $j=0,1,2$
\begin{equation}\label{3.371} 
\forall y\in \mathbb{R}^n_- :\qquad 
|D^{2+j} \psi (y) | \le C |y|^{2-n-j},\quad\mbox{\ where\ } C=C(x). 
\end{equation} 
As above $\psi^*$ is an entire  biharmonic function and so are $D\psi^*$ and 
$D^2\psi^*$. Since $|D^2\psi^*(y)|\le C(1+|y|)^{2-n}$, it follows that 
$D^2\psi^*(x)\equiv 0$. In view of the boundary 
conditions in (\ref{3.35}) we come up with $\psi^*(y) \equiv 0$ 
also in the case $n=3,4$.

It remains to prove (\ref{3.37}) and (\ref{3.371}).
We consider first $n>4$.
Assume by contradiction that there exists a sequence $(y_{\ell})\subset\mathbb{R}^n_-$
such that $|\nabla^j \psi (y_{\ell} ) | \cdot |y_{\ell}|^{n+j-4}\to \infty$ for $\ell \to\infty$.
Then
$$
\tilde{\psi}_{\ell} (y) :=|y_{\ell}|^{n-4}\psi\left( y_{\ell}-y_{\ell,1}\vec{e}_1+|y_{\ell}|y\right)
$$
would solve
\begin{equation}\label{3.38}
\left\{ \begin{array}{ll}
\Delta^2 \tilde{\psi}_{\ell} =0 &\mbox{\ in\ } \mathbb{R}^n_-,\\
\tilde{\psi}_{\ell} =\partial_1 \tilde{\psi}_{\ell} =0&\mbox{\ on\ }\{ y_1 =0\}.
\end{array}\right.
\end{equation}
{From}  the assumption we conclude that
\begin{equation}\label{3.39}
\left|\nabla^j \tilde{\psi}_{\ell}\left(\frac{y_{\ell,1}}{|y_{\ell}|}\vec{e}_1 \right)\right|
=|y_{\ell}|^{n+j-4}|\nabla^j \psi (y_{\ell} ) |\to\infty.
\end{equation}
On the other hand,
\begin{equation}\label{3.391}
|\tilde{\psi}_{\ell} (y)| \le C |y_{\ell}|^{n-4}\left| y_{\ell}-y_{\ell,1}\vec{e}_1+|y_{\ell}|y\right|^{4-n}
\le C\left|     \frac{y_\ell}{|y_\ell|} +y- \frac{y_{\ell,1}}{|y_\ell|}\vec{e}_1\right|^{4-n},
\end{equation}
so that $\tilde{\psi}_{\ell}$ remains bounded in a neighbourhood
of $\frac{y_{\ell,1}}{|y_\ell|}\vec{e}_1$ in
$\overline{\mathbb{R}^n_-}$. Local Schauder estimates
\cite[Theorem 7.3]{ADN} yield
$$
\left|\nabla^j \tilde{\psi}_{\ell}\left(\frac{y_{\ell,1}}{|y_{\ell}|}\vec{e}_1 \right)\right|\le C,
$$
thereby contradicting (\ref{3.39}). This proves (\ref{3.37}).

As for (\ref{3.371}), i.e. in particular $n=3,4$,
the proof is quite similar since we can already make
use of the gradient estimates (\ref{3.361}). Instead of (\ref{3.391}) one has
to make use of
$$
|\nabla \tilde{\psi}_{\ell} (y)| \le
C |y_{\ell}|^{n-3}\left| y_{\ell}-y_{\ell,1}\vec{e}_1+|y_{\ell}|y\right|^{3-n}
\le C\left|     \frac{y_\ell}{|y_\ell|} +y- \frac{y_{\ell,1}}{|y_\ell|}\vec{e}_1\right|^{3-n},
$$
so  that
$\nabla\tilde{\psi}_{\ell}$ remains bounded uniformly outside
$\frac{y_\ell}{|y_\ell|}- \frac{y_{\ell,1}}{|y_\ell|}\vec{e}_1$.
Therefore, since $\tilde{\psi}_{\ell}$ vanishes on
$\partial\mathbb{R}_-^n$, we get that $\tilde{\psi}_{\ell}$ is
bounded in a neighbourhood of
$\frac{y_{\ell,1}}{|y_\ell|}\vec{e}_1$ in
$\overline{\mathbb{R}^n_-}$. The proof  of the present lemma is
complete.
\qed\end{proof}

In order to show that the present case $x_\infty=y_\infty\in\partial\Omega$
cannot occur we collect some basic facts on the biharmonic Green's
function in the half space; modulo a simple conformal transformation, cf.~\cite[p. 126]{Boggio}:

\begin{lemma}\label{lemma3.6}
The biharmonic Green's  function in $\mathbb{R}^n_-$ is given by
\begin{equation}\label{3.45}
\forall x,y\in \mathbb{R}^n_-: H(x,y)= \frac{1}{4ne_n} \, |x-y|^{4-n}
\int\limits_1^{ \left|  x^*-y  \right|  / |x-y| }
    (v^2 -1) v^{1-n} \, dv,
\end{equation}
where $x^*=(-x_1,\bar{x})$. {From} this
it follows by direct calculation:
\begin{eqnarray}
\forall x,y\in \mathbb{R}^n_-, \ x\not=y:&&\quad H(x,y) >0;
\label{Green1}\\
\forall x\in \mathbb{R}^n_-,y\in\partial  \mathbb{R}^n_-:&&\quad  \Delta_y H(x,y)>0;
\label{Green2}\\
\forall x,y\in \partial\mathbb{R}^n_-, \ x\not=y:&&\quad\Delta_x  \Delta_y H(x,y) >0.
\label{Green3}
\end{eqnarray}
\end{lemma}

We proceed by showing that $x_\infty=y_\infty\in\partial\Omega$ is indeed impossible
and recall that by assumption we chose $x_k,y_k$ such that $G_k( x_k,y_k)=0$.
In terms of the transformed Green's  functions this reads
\begin{equation}\label{3.55}
 \tilde{G}_k \left( \rho_k\vec{e}_1,  \frac{y_k'-x_k'}{|x_k'-y_k'|}+ \rho_k\vec{e}_1\right)=0,
\end{equation}
cf. (\ref{3.26}).
After possibly extracting a further subsequence we find a point
$$
\theta=\lim_{k\to\infty} \frac{y_k'-x_k'}{|x_k'-y_k'|}
$$
and may conclude that
\begin{equation}\label{3.56}
\tilde{G} \left(\rho\vec{e}_1,\theta+\rho\vec{e}_1 \right)=0.
\end{equation}
According to the possible location of the limit points we have to distinguish
four cases:

\vspace{2mm} 
\noindent 
{\it Case (a): $\rho<0$ and $\left( \theta+\rho\vec{e}_1 \right)_1<0$.} 
We put $\tilde{x}=(\sigma\circ L) (\rho\vec{e}_1)\in \mathbb{R}^n_-$, 
$\tilde{y}=(\sigma\circ L) (\theta+\rho\vec{e}_1)\in \mathbb{R}^n_-$. According
to  (\ref{3.32}) and Lemma~\ref{lemma3.5} we could conclude that 
$$ 
H(\tilde{x},\tilde{y})= \bar{G} (\tilde{x},\tilde{y}) =0, 
$$ 
which is impossible in view of (\ref{Green1}). 
 
\vspace{2mm} 
\noindent 
{\it Case (b): $\rho=0$ and $\left( \theta+\rho\vec{e}_1 \right)_1<0$.} 
As in the proof of Lemma~\ref{lemma3.2} we conclude from (\ref{3.55}) 
that $\partial_{x_1}^2 \tilde{G} (0,\theta)=0$. Together with the 
Dirichlet boundary conditions satisfied by $\tilde{G} $ this yields 
$\tilde{G} (0,\theta)=0, D_x \tilde{G} (0,\theta)=0, D_x^2\tilde{G} 
(0,\theta)=0$. If we put $\tilde{y}=(\sigma\circ L) (\theta)\in\mathbb{R}^n_-$ 
this implies due 
to (\ref{3.32}) that also $D_x^2\bar{G} (0,\tilde{y})=0$. In particular, we
have  that $\Delta_x H (0,\tilde{y})=\Delta_x \bar{G} (0,\tilde{y})=0$, 
which  is impossible in view of 
(\ref{Green2}). 
 
\vspace{2mm} 
\noindent 
{\it Case (c): $\rho<0$ and $\left( \theta+\rho\vec{e}_1 \right)_1=0$.} 
Due to symmetry of the Green's  function, this case is completely analogous 
to the previous one and hence impossible in view of (\ref{Green2}). 
 
\vspace{2mm} 
\noindent 
{\it Case (d): $\rho=0$ and $\left( \theta+\rho\vec{e}_1 \right)_1=0$.} 
As in the proof of Lemma~\ref{lemma3.4} we conclude from (\ref{3.55}) 
that $\partial_{x_1}^2 \partial_{y_1}^2 \tilde{G} (0,\theta)=0$. Here 
$\theta_1=0, |\theta|=1$. Thanks to the boundary conditions satisfied 
by $\tilde{G} $ this gives $\forall |\alpha|\le 2, |\beta|\le 2: 
\quad D^{\alpha}_x D^{\beta}_y \tilde{G} (0,\theta)=0$. 
Using again  (\ref{3.32}), we see that also $\forall |\alpha|\le 2, 
|\beta|\le 2: \quad D^{\alpha}_x D^{\beta}_y \bar{G} (0,\tilde{y})=0$, where 
$ \tilde{y}=(\sigma\circ L) (\theta)\not=0$. In particular, we come up with 
$\Delta_x\Delta_y H(0,\tilde{y})=\Delta_x\Delta_y \bar{G} (0,\tilde{y})=0$. 
This is impossible in view of (\ref{Green3}).

\bigskip
{\it Conclusion.} In each case  we finally deduced a contradiction
so that $x_\infty=y_\infty\in\partial\Omega$ is indeed impossible.
The proof of Lemma~\ref{lemma3.3} is complete.
\hfill $\square$

\subsection{Proof of Theorems~\ref{smallnegativepart},
\ref{theorem1} and \ref{theorem2}}

Theorem~\ref{theorem2} follows from the conclusions made in Subsections~\ref{interiorinterior},
\ref{interiorboundary} and \ref{boundaryboundary}. 

In order to prove Theorem~\ref{smallnegativepart} we assume for
contradiction that there exist a bounded $C^{4,\alpha}$-smooth domain
$\Omega\subset\mathbb{R}^n$ and sequences $(x_k), (y_k) \subset \Omega$,
$x_k\not=y_k$ with $H_{\Omega} (x_k,y_k) \le 0$ and
$\lim_{k\to\infty} |x_k-y_k|=0$. In view of the smoothness assumption,
we see by working in local coordinate charts that after possibly
passing to a subsequence and relabelling we
find $\tilde{y}_k\in\Omega$, $x_k\not=\tilde{y}_k$
with $H_{\Omega} (x_k,\tilde{y}_k) = 0$ and $|x_k-\tilde{y}_k | \to 0$ for $k\to\infty$. 
Application of Theorem~\ref{theorem2}
in the special case $\Omega_k=\Omega$, $a_k=0$ shows that this is impossible.
This contradiction proves that there exists a $\delta=\delta(\Omega)>0$ 
such that $x,y\in\Omega$, $x\not=y$, $H_{\Omega} (x,y) \le 0 \
\Rightarrow \ |x-y|\ge \delta$. Estimate \eqref{boundfrombelow0} now follows
directly  from (\ref{4.1}) while \eqref{boundfrombelow} is a consequence of
(\ref{twosidedestimate}), i.e. of  DallAcqua and Sweers \cite{DallacquaSweersJDE}.

In order to prove Theorem~\ref{theorem1},
 we assume that no such $\varepsilon_0>0$ exists. In view of the remark 
after Theorem~\ref{theorem1},
we would have a neighbourhood $U$ of $\overline{B}$, $C^{4,\alpha}$-smooth diffeomorphisms
$\psi_k:U\to \psi_k (U)$ and smooth domains $\Omega_k=\psi_k (B)$ with
sign changing biharmonic Green's  functions $H_k$. Hence, one of the alternatives described in
Theorem~\ref{theorem2} would occur for the  biharmonic Green's  function $H$ in the ball $B$.
Since $H$ enjoys precisely the analogous properties of Lemma~\ref{lemma3.6}
(cf. \cite[p. 126]{Boggio}), this is false;
Theorem~\ref{theorem1} follows.\hfill $\square$


\end{document}